\documentclass[11pt]{amsart}

\usepackage{epsfig}

\usepackage{graphicx}
\usepackage{amscd}
\usepackage{amsmath}
\usepackage{amsfonts}
\usepackage{amssymb}
\usepackage{verbatim}
\usepackage{color}
\usepackage{mathrsfs}

\newtheorem{theorem}{Theorem}[section]
\theoremstyle{plain}

\newtheorem{corollary}[theorem]{Corollary}

\newtheorem{defi}[theorem]{Definition}

\newtheorem{lemma}[theorem]{Lemma}

\newtheorem{prop}[theorem]{Proposition}
\newtheorem{remark}[theorem]{Remark}

\numberwithin{equation}{section}

\def\Ns{{\mathscr N}}

\newcommand \A {{\cal A}}

\newcommand \T {{\cal T}}

\newcommand \q {{\bf q}}

\newcommand \G {{\bf G}}

\def\ba{\bom}

\def\FrA{{\mathfrak A}}
\def\FrR{{\mathfrak R}}

\def\Aa{{\mathbb A}}

\def\Lip{{\rm Lip}}

\def\Xx{{\mathfrak X}}

\def\ef{{\mathfrak e}}

\def\gf{{\mathfrak g}}
\def\One{{1\!\!1}}

\def\Hk{{\mathcal H}}
\def\Xk{{\mathcal X}}

\def\wtil{\widetilde}
\def\wt{\wtil}

\def\half{\frac{1}{2}}

\newcommand{\lam}{\lambda}

\newcommand{\gam}{\gamma}
\newcommand{\om}{\omega}
\newcommand{\omb}{\boldsymbol \omega}

\def\Om{\Omega}
\newcommand{\Sig}{\Sigma}

\newcommand{\Gam}{\Gamma}
\newcommand{\sig}{\sigma}

\def\AA{{\mathbb A}}

\newcommand{\R}{{\mathbb R}}

\newcommand{\Z}{{\mathbb Z}}
\newcommand{\C}{{\mathbb C}}

\def\N{{\mathbb N}}

\newcommand{\Prob}{{\mathbb P}\,}
\def\P{\Prob}

\def\bp{{\bf p}}
\def\bq{{\bf q}}
\def\bw{{\bf w}}

\def\bqcyl{[\bq.\bq]}
\def\A{{\mathcal A}}
\def\Bk{{\mathcal B}}
\def\Rk{{\mathcal R}}
\def\Qk{{\mathcal Q}}

\def\Pk{{\mathcal P}}

\def\Sf{{\sf S}}
\def\T{{\mathbb T}}
\def\Vk{{\mathcal V}}
\def\Zk{{\mathcal Z}}

\def\be{\begin{equation}}
\def\ee{\end{equation}}
\newcommand{\Ek}{{\mathcal E}}
\newcommand{\Fk}{{\mathcal F}}

\newcommand{\eps}{{\varepsilon}}

\newcommand{\es}{\emptyset}

\newcommand{\HH}{{\mathcal H}}

\def\card{{\rm card}}

\def\ve1{\vec{1}}
\def\Fre{{\mathfrak E}}
\def\Frs{{\mathfrak S}}
\def\Frg{{\mathfrak G}}
\def\ok{{\mathfrak o}}

\def\Ak{{\mathcal A}}

\def\bom{{\bf a}}
\def\eb{{\bf e}}

\begin{document}

\title [H{\"o}lder regularity for the spectrum of translation flows]{H{\"o}lder regularity for the spectrum of  translation flows}

\author{Alexander I. Bufetov}
\address{Alexander I. Bufetov,
Aix-Marseille Universit{\'e}, CNRS, Centrale Marseille, I2M, UMR 7373,  39 rue F. Joliot Curie Marseille France}
\address{Steklov  Mathematical Institute of RAS, Moscow} \address{Institute for Information Transmission Problems, Moscow}

\email{bufetov@mi.ras.ru}

\author{Boris Solomyak }
\address{Boris Solomyak\\ Department of Mathematics, Bar-Ilan University, Ramat-Gan, Israel}
\email{bsolom3@gmail.com}

\begin{abstract} The paper is devoted to generic translation flows corresponding to Abelian differentials on flat surfaces of arbitrary genus $g\ge 2$. These flows are weakly mixing by the Avila-Forni theorem.
In genus 2, the H\"older property for the spectral measures of these flows  was
 established in  \cite{BuSo18a,BuSo19}.
Recently Forni \cite{Forni2}, motivated by \cite{BuSo18a}, obtained H\"older estimates for spectral measures in the case of surfaces of 
arbitrary genus.  Here we combine Forni's idea with  the symbolic approach of  \cite{BuSo18a} and 
prove  H\"older regularity for spectral measures of flows on random Markov compacta, in particular, for translation flows  
for an arbitrary genus $\ge 2$.
\end{abstract}

\maketitle

\section{Introduction}
\subsection{Formulation of the main result}

Let $M$ be a compact orientable surface.  To
  a holomorphic one-form $\omb$ on $M$ one can assign  the corresponding {\it vertical} flow $h_t^+$ on $M$, i.e., the flow at unit speed along the
 leaves of the foliation $\Re(\omb)=0$.
The vertical flow preserves the measure 
$
{\mathfrak m}=i(\omb\wedge {\overline \omb})/2, 
$
the area form induced by $\omb$. 
By a theorem of Katok \cite{katok}, the flow $h_t^+$ is never mixing.
The moduli space of abelian differentials carries a natural volume measure, called the Masur-Veech measure \cite{masur}, \cite{veech}.  
For almost every Abelian differential with respect to the Masur-Veech measure,  Masur \cite{masur} and Veech \cite{veech}  independently and simultaneously proved that the flow   $h_t^+$ is
uniquely ergodic.   Weak mixing for almost all  translation flows has been established by Veech in \cite{veechamj} under additional assumptions on the combinatorics of the abelian differentials and  by Avila and Forni
\cite{AF} in full generality. The spectrum of translation flows is therefore almost surely  continuous and always has a singular component. 

\thispagestyle{empty}

Sinai [personal communication] raised the question: to find the local asymptotics for the spectral measures of translation flows.
In \cite{BuSo18a,BuSo19} we developed an approach to  this problem and succeeded in obtaining H\"older estimates for spectral measures in the case of surfaces of genus 2.
The proof proceeds via uniform estimates of twisted Birkhoff integrals
in the symbolic
framework of random Markov compacta and arguments of Diophantine nature in the spirit of Salem, Erd\H{o}s and Kahane.
 Recently  Forni \cite{Forni2} obtained H\"older estimates for spectral measures  in the case of surfaces of arbitrary genus. While  Forni does not use the symbolic formalism, the main idea of his approach can also be formulated in symbolic terms: namely, that instead of the {\it scalar} estimates of \cite{BuSo18a,BuSo19}, we can use the Erd\H{o}s-Kahane argument in {\it vector} form, cf. \eqref{lattice} et seq. 
Following the idea of Forni and directly using  the vector form of the Erd\H{o}s-Kahane argument yields a considerable simplification of our initial proof and allows us to prove the H\"older  property for a general class of random Markov compacta, cf. \cite{Buf-umn}, and, in particular, for almost all translation flows on surfaces of arbitrary genus. 

Let $\HH$ be a stratum of abelian differentials on a surface of genus $g\ge 2$. The  natural smooth Masur-Veech measure on the stratum $\HH$ is denoted by $\mu_\HH$.
Our main result is
that for almost all abelian differentials in $\HH$, the spectral measures of Lipschitz functions with respect to the corresponding translation flows have the H{\"o}lder property. 
Recall that for a square-integrable test function $f$, the spectral measure $\sig_f$ is defined by
$$
\widehat{\sig}_f(-t) = \langle f\circ h_t^+, f \rangle,\ \ \ t\in \R,
$$
see Section~\ref{sec-twist}.
A point mass for the spectral measure corresponds to an eigenvalue, so H\"older estimates for spectral measures quantify weak mixing for our systems.


\begin{theorem}\label{main-moduli}
There exists $\gamma>0$ such that  for $\mu_\HH$-almost every abelian differential $(M, \omb)\in \HH$ the following holds. For any $B>1$ there exist constants $C=C(\omb,B)$ and $r_0=r_0(\omb,B)$ such that  for any Lipschitz function $f$ on $M$, for all $\lam\in [B^{-1},B]$  we have 
\be \label{eq-moduli}
\sig_f([\lam-r, \lam+r])\le C\|f\|_L\cdot r^\gamma\ \ \mbox{for all} \ r\in (0, r_0).
\ee
\end{theorem} 

This theorem is analogous to Forni \cite[Corollary 1.7]{Forni2}.

\begin{remark} \label{remark-main}
{\em Our argument, as well as Forni's, see \cite[Remark 1.9]{Forni2}, remains valid for  almost every translation flow under a more general class of measures. Let $\mu$ be a Borel probability measure invariant and ergodic under the Teichm{\"u}ller flow. 
  Let $\kappa$ be the number of positive Lyapunov exponents for the Kontsevich-Zorich cocycle under the measure $\mu$.
 To formulate our condition precisely, recall that,  by the Hubbard-Masur theorem on the existence of cohomological coordinates,  the moduli space of abelian differentials with prescribed singularities can be locally identified with the space 
 ${\widetilde H}$ of relative cohomology, with complex coefficients, of the underlying surface with respect to the singularities.  Consider the subspace $H\subset {\widetilde H}$ corresponding to the absolute cohomology, the corresponding fibration of ${\widetilde H}$ into translates of $H$ and its image, a fibration $\overline {\mathcal F}$ on the moduli space of abelian differentials with prescribed singularities. Each fibre is locally isomorphic to $H$ and thus has dimension equal to $2g$, where $g$ is the genus of the underlying surface. We now restrict ourselves to the subspace of abelian differentials of area $1$, and let the fibration $\mathcal F$ be the restriction  of the fibration $\overline {\mathcal F}$; the dimension of each fibre of the fibration $\mathcal F$ is equal to $2g-1$.  Almost every fibre of $\mathcal F$ carries a conditional measure, defined up to multiplication by a constant,  of the measure $\mu$. If there exists 
 $\delta>0$ such that the Hausdorff dimension of the conditional measure of $\mu$ on almost every  fibre of $\mathcal F$ has Hausdorff dimension at least $2g-\kappa+\delta$, then Theorem \ref{main-moduli} holds for $\mu$-almost every abelian differential.
 
 In the case of the Masur-Veech measure $\mu_\HH$, it is well-known that the conditional measure on almost every fibre is mutually absolutely continuous with the Lebesgue measure, hence has Hausdorff dimension $2g$. By the celebrated result of Forni \cite{Forni}, there are $\kappa=g$ positive Lyapunov exponents for the Kontsevich-Zorich cocycle under the measure $\mu_\HH$, so 
 Theorem~\ref{main-moduli} will follow by taking any $\delta\in (0,1)$.
 }
 \end{remark}

 \medskip
 
The proof of the H\"older property for spectral measures proceeds via upper bounds on the growth of twisted Birkhoff integrals
$$
S_R^{(x)}(f,\lam) := \int_0^R e^{-2\pi i \lam \tau} f\circ h^+_\tau(x)\,d\tau.
$$

\begin{theorem} \label{th-twisted}
There exists $\alpha\in (0,1)$ such that for $\mu_\HH$-almost every abelian differential $(M, \omb)\in \HH$ and any $B>1$ there exist $C'=C'(\omb,B)$ and $R_0 = R_0(\omb,B)$ such that for any Lipschitz function $f$ on $M$, for all $\lam\in [B^{-1},B]$ and all $x\in M$, 
\begin{equation} \label{au-twisted}
\left|S_R^{(x)}(f,\lam)\right| \le C' R^\alpha\ \ \mbox{for all}\ R\ge R_0.
\end{equation}
\end{theorem}

This theorem is analogous to Forni \cite[Theorem 1.6]{Forni2}.
The derivation of Theorem~\ref{main-moduli} from Theorem~\ref{th-twisted} is standard, with $\gam = 2(1-\alpha)$; see Lemma~\ref{lem-varr}. In fact, in order to obtain (\ref{eq-moduli}), $L^2$-estimates (with respect to 
the area measure on $M$) of $S_R^{(x)}(f,\lam)$ suffice; we obtain bounds that are uniform in $x\in M$, which is of independent interest.

{
\begin{remark}
Using the approach of \cite{BuSo18b}, we expect that we can make the estimate in \eqref{eq-moduli} uniform
on the entire real line.
Such uniform estimates imply results on quantitative weak mixing in the form
\begin{equation} \label{eq-parab}
\frac{1}{R} \int_0^R \left|\langle f\circ h_t^+, g\rangle\right|\,dt \le C(f,g) \cdot R^{-\beta},
\end{equation}
for some $\beta>0$ and $f,g$ of the appropriate Lipschitz class, see \cite[Theorem 1.1]{BuSo18b}. 
Forni \cite[Corollary 1.8]{Forni2} derived a result similar to  \eqref{eq-parab} from the estimates of twisted Birkhoff integrals in the form \eqref{au-twisted} using  Fourier transform techniques, see
\cite[Lemma 9.3]{Forni2}, under the additional assumption that $f$ is sufficiently smooth.
\end{remark}
}

\subsection{Quantitative weak mixing} There is a close relation between H\"older regularity of spectral measures and quantitative rates of weak mixing, see \cite{BuSo18b,Forni2}. One can also note a connection of our arguments with the proofs of weak mixing, via Veech's criterion \cite{veechamj}. A translation flow can be represented, by considering its return map to a transverse interval, as a suspension flow over an interval exchange transformation (IET) with a roof function constant on each sub-interval. The roof function is determined by a vector $\vec s\in H\subset \R^m_+$ with positive coordinates, where $m$ is the number of sub-intervals of the IET and $H$ is a subspace of dimension $2g$ corresponding to the space of absolute cohomology from Remark 1.2.  Let $\AA(n,\bom)$ be the Zorich acceleration of the Masur-Veech cocycle on $\R^m$, corresponding to returns to a ``good set'', where $\bom$ encodes the IET. Veech's criterion \cite[\S 7]{veechamj} says that if 
$$
\limsup_{n\to \infty} \|\AA(n,\bom)\cdot \om \vec{s}\|_{\R^m/\Z^m}>0\ \ \mbox{for all}\ \om\ne 0,
$$
then the translation flow corresponding to $\vec s$ is weakly mixing. This was used by Avila and Forni \cite[Theorem A.2]{AF} to show that for almost every $\bom$ the set of $\vec s\in H$, such that the suspension flow is not weakly mixing, has Hausdorff dimension at most $g+1$. On the other hand, our Proposition~\ref{prop-quant} says that if the set
$$
\{n\in \N:\ \|\AA(n,\bom)\cdot \om \vec{s}\|_{\R^m/\Z^m} \ge \varrho\}
$$
has positive lower density in $\N$ for some $\varrho>0$ uniformly in $\om\ne 0$ bounded away from zero and from infinity, then spectral measures corresponding to Lipschitz functions have the H\"older property. The Erd\H{o}s-Kahane argument is used to estimate the dimension of those $\vec s$ for which this fails. 
In recent Forni's work, a version of the weak stable space for the Kontsevich-Zorich cocycle, denoted $W_K^s(h)$ and defined in \cite[Section 6]{Forni2}, seems analogous to our exceptional set $\Fre$ defined in (\ref{def-excep}). The Hausdorff dimension of this set is estimated in \cite[Theorem 6.3]{Forni2}, with the help of \cite[Lemma 6.2]{Forni2}, which plays a r\^ole similar to our Erd\H{o}s-Kahane argument.

\subsection{Organization of the paper, comparison with \cite{BuSo18a} and \cite{BuSo19}} A large part of the paper \cite{BuSo18a} was written in complete generality, without the genus 2 assumptions, and is directly used  here; for 
example, we do not reproduce the estimates of twisted Birkhoff integrals using generalized matrix Riesz products, but refer the reader to \cite[Section 3]{BuSo18a} instead.
 Sharper estimates of matrix Riesz products were obtained in \cite{BuSo19}, where the matrix Riesz products products are interpreted in terms of the {\em spectral cocycle}. In particular, we established  a formula relating the local upper Lyapunov exponents of the cocycle with the pointwise lower dimension of spectral measures. Note nonetheless that the cocycle structure is not used in this paper.
Section 3, parallel to \cite[Section 4]{BuSo18a}, contains the main result on H\"older regularity of spectral  measures in the setting of random $S$-adic, or, equivalently, random Bratteli-Verhik systems.
The main novelty is that here we only require that the second Lyapunov exponent $\theta_2$ of the 
Kontsevich-Zorich cocycle be positive, while in \cite{BuSo18a} the assumption was that $\theta_1 > \theta_2 >0$ are the {\em only} non-negative exponents. 
The preliminary Section 4 closely follows  \cite{BuSo18a}.
The crucial changes occur in Sections 5 and 6, where, in contrast with   \cite{BuSo18a},  Diophantine approximation is established {\it in  vector form}, cf. Lemma~\ref{lem-lattice}. The exceptional set is defined in (\ref{def-excep}), and the Hausdorff dimension of the exceptional set is estimated in Proposition~\ref{prop-EK}. Although the general strategy of the ``Erd\H{o}s-Kahane argument'' remains, the implementation is now significantly simplified. In \cite{BuSo18a} we worked with scalar parameters, the coordinates of the vector of heights with respect to the Oseledets basis, but here we simply consider the vector parameter and work with the projection of the vector of heights to the strong unstable subspace.
In particular,  the cumbersome estimates of \cite[Section 8]{BuSo18a} are no longer needed.
Section 7, devoted to the derivation of the main theorem on translation flows from its symbolic counterpart, parallels \cite[Section 11]{BuSo18a} with some changes. The most significant one is that we require a stronger property of the good returns, which is achieved in Lemma~\ref{lem-combi1}. On the other hand, the large deviation estimate for the Teichm\"uller flow required in Theorem~\ref{th-main1} remains unchanged, and we directly use  \cite[Prop.\,11.3]{BuSo18a}.

\subsection{Further directions} As mentioned in the abstract, the result for translation flows is a special case of a theorem on H\"older regularity for spectral measures of flows on random Markov compacta (see the next section for definitions).
Random two-sided Markov compacta endowed with a Vershik ordering and ergodic flows along their stable foliations were studied  in \cite{Buf-umn}.
 This is a very general symbolic framework, which offers a possibility of other applications. In particular, this includes random substitution tilings on the line and suspension flows over random $S$-adic systems.
{\em Self-similar} substitution tilings have also been studied extensively in higher dimensions. In \cite{BuSol--limit} we obtained limit theorems for the deviation of ergodic averages for self-similar tiling $\R^d$-actions. It would be very interesting also to obtain H\"older estimates for the spectrum of such systems. For {\em random} substitution $\R^d$-actions such estimates were recently obtained by 
Trevi\~no \cite{Tre_prepint}, who used a generalization of our symbolic approach. In another direction,
Lindsey and Trevi\~no \cite{LinTre} constructed a wide range of flat surfaces of infinite genus, but finite area, using bi-inifnite ordered Bratteli diagrams. It would be of interest to apply our results  to translation flows on surfaces of infinite genus  (we are grateful to Rodrigo Trevi\~no for this remark). 

\medskip

\noindent {\bf{Acknowledgements.}} We are deeply grateful to Giovanni Forni for generously sharing his ideas with us and for sending us his manuscript containing his proof of the H\"older property in arbitrary genus. We would like to thank Corinna Ulcigrai for her hospitality in Bristol (B. S.) and Z\"urich (A. B. and B. S.), and for many fruitful discussions. {\ We are grateful to the anonymous referees for a careful reading of the paper and many helpful critical comments.}

A. B.'s research  is supported by the European Research Council (ERC) under the European Union Horizon 2020 research and innovation programme, grant  647133 (ICHAOS), by the Agence Nationale de Recherche,
 project ANR-18-CE40-0035, and by the Russian Foundation for  Basic Research, grant 18-31-20031. B. S.'s research is supported by the Israel Science Foundation (ISF), grants 396/15 and 911/19.


\section{Preliminaries}

\subsection{Markov compacta}
The symbolic representation of translation flows was given by Bufetov \cite{Buf-umn}, using the theory of  Markov compacta. Here we briefly recall the definitions.

Denote by $\Frg$ the set of all oriented graphs on $m$ vertices such that there is an edge starting at every vertex and an edge terminating at every vertex (we allow loops and multiple edges).
For $\Gam\in \Frg$ let $\Ek(\Gam)$ be the set of edges of $\Gam$. For $e\in \Ek(\Gam)$ denote by $I(e)$ its initial vertex and by $F(e)$ its terminal vertex. Let $A(\Gam)$ be the incidence matrix of $\Gam$ given by
$$
A_{ij}(\Gam) = \#\{e\in \Ek(\Gam): I(e) = i, F(e) = j\}.
$$
Assume that we are given a sequence $\gf = \{\Gam_n\}_{n\in \Z}$, with $\Gam_n \in \Frg$. The {\em Markov compactum} of paths in the sequence of graphs $\gf$ is defined by
$$
X = X(\gf) = \{\ef = (\ef_n)_{n\in \Z}:\ \ef_n \in \Ek(\Gam_n),\ F(\ef_{n+1}) = I(\ef_n)\}.
$$
We will also need the one-sided Markov compacta $X_+$ (respectively $X_-$), defined in the same way with elements $(\ef_n)_{n\ge 1}$ (respectively $(\ef_n)_{n\le 0}$). A one-sided sequence of graphs in $\Frg$ is also called a {\em Bratteli diagram} of rank $m$.
For $\ef\in X,\ n\in \Z^m$ introduce the sets 
$$
\gam_n^+(\ef) = \{\ef'\in X:\ \ef_j'=\ef_j, j\ge n\}:\ \ \ 
\gam_n^-(\ef) = \{\ef'\in X:\ \ef_j'=\ef_j, j\le n\};
$$
$$
\gam_\infty^+(\ef) = \bigcup_{n\in \Z} \gam_n^+(\ef);\ \ \ \gam_\infty^-(\ef) = \bigcup_{n\in \Z} \gam_n^-(\ef).
$$
The sets $\gam_\infty^+(\ef)$ are leaves of the {\em asymptotic foliation} $\Fk^+(X)$ on $X$ corresponding to the infinite future, and the sets $\gam_\infty^-(\ef)$
are the leaves of the asymptotic foliation $\Fk^-(X)$ on $X$ corresponding to the infinite past.

There is a standard construction of {\em telescoping} (= aggregation) {\ \cite[2.4]{Buf-umn}}: for any sequence $1=n_0<n_1 < n_2< \cdots$ we ``concatenate'' the graphs
$\Gam_{n_j},\ldots, \Gam_{n_{j+1}-1}$ to obtain $\Gam'_j\in \Frg$. Telescoping induces an equivalence relation on $\Frg^\Z$, and the Markov compacta corresponding to equivalent $\gf$ and $\gf'$ are naturally isomorphic.

\medskip

\noindent
{\bf Standing Assumption.}  {\em The sequence $\{\Gam_n\}$ (after appropriate
telescoping) contains infinitely many occurrences of a single graph $\Gam$ with a strictly positive incidence matrix, both in the past and in the future. }

\medskip

In this case, as is
well-known since the work   of Furstenberg (see e.g.\ (16.13) in \cite{furst}), the 
Markov compactum $X$ is {\em uniquely ergodic}, which means that there are unique probability measures $\nu_+$, $\nu_-$, invariant under the equivalence relations defined by the future (past) asymptotic foliations respectively.

\subsection{Vershik's orderings and Vershik automorphisms}
The Markov compactum and the asymptotic foliations encode the translation surface and its vertical and horizontal foliations; however, in order to recover the translation flows themselves, one needs a linear ordering on the leaves of the foliation. This linear ordering is induced by Vershik's orderings of the edges of the graphs defining the Markov compactum.

Formally, following Ito \cite{Ito} and Vershik \cite{Vershik1,Vershik2}, we assume that a linear ordering (called {\em Vershik's ordering}) is given on the set $\{e\in \Ek(\Gam_n): I(e)=i\}$ for every graph $\G_n$ and for every vertex $i$. This induces a linear ordering on any leaf of the foliation $\Fk_X^+$. Indeed, if $\ef' \in \gam_\infty^+(\ef)$, $\ef'\ne \ef$, then there exists $n$ such that $\ef_j = \ef_j'$ for $j>n$ but $\ef_j \ne \ef_j'$. Since $I(\ef_n) = I(\ef_n')$, the edges $\ef_n$ and $\ef_n'$ are comparable with respect to our ordering; if $\ef_n < \ef_n'$, then we write $\ef< \ef'$.
Denote the resulting ordering on $\Fk_X^+$ by $\ok$.
Restricting this ordering to the 1-sided compactum $X_+$, we obtain the {\em adic}, or {\em Vershik automorphism} 
${\mathfrak T}$, defined as the immediate successor of a path $\ef$ in the ordering $\ok$. This is a $\Z$-action on the complement of the countable set consisting of the union of finitely many one-sided orbits of the maximal and minimal paths in the ordering $\ok$. In the literature, the Vershik automorphism is often called a {\em Bratteli-Vershik dynamical system}.
By Vershik's Theorem \cite{Vershik1,Vershik2}, every ergodic automorphism of a Lebesgue space is measurably isomorphic to the Vershik automorphism on a one-sided ordered Bratteli-Vershik diagram (in general, of infinite rank). A  realization of a (minimal) IET as a Vershik automorphism (similar to the one described below, see \eqref{induc}, via the Rauzy induction) was given by Gjerde and Johansen \cite{GJ}. On the other hand, given a Vershik's ordering on the 2-sided uniquely ergodic Markov compactum $X$, Bufetov \cite{Buf-umn}, defined a flow on $X$, which is isomorphic to a suspension flow over the Vershik automorphism, with a piecewise-constant roof function. This is the construction which yields the symbolic representation of translation flows.

\subsection{The space of Markov compacta and the renormalization cocycle} Let $\Om=\Frg^\Z$ be the space of bi-infinite sequences of graphs $\Gam_n \in \Frg$, with the left shift
$\sig$. We have a natural cocycle $\Aa$ over the dynamical system $(\Om,\sig)$ defined, for $n>0$ by the formula:
$$
\Aa(n,\gf) = A(\Gam_n)\ldots A(\Gam_1),\ \ \mbox{where}\ \gf = \{\Gam_n\}_{n\in \Z}\in \Om.
$$
Let $\Om_{inv}$ be the subset of all sequences $\gf$ such that all matrices $A(\Gam_n)$ are invertible. For $\gf\in \Om_{inv}$ and $n<0$ we set
$$
\Aa(n,\gf) = A^{-1}(\Gam_{-n})\cdots A^{-1}(\Gam_0),
$$
and let $\Aa(0,\gf)$ be the identity matrix.


\subsection{Substitutions and $S$-adic systems} \label{sec-Sadic}
Along with Markov compacta and Vershik automorphisms, it is convenient to use the language of {substitutions} (see \cite{Queff,Fogg} for background). Consider the
alphabet $\Ak$, and denote by $\Ak^+$ the set of finite (non-empty) words with letters in $\Ak$. A {\em substitution}  is a map $\zeta:\,\A\to \A^+$, extended to  $\A^+$ and $\A^{\N}$ by concatenation. 
The {\em substitution matrix} is defined by 
\be \label{sub-mat}
\Sf_\zeta (i,j) = \mbox{number of symbols}\ i\ \mbox{in the word}\ \zeta(j).
\ee
Denote by $\FrA$ the set of substitutions $\zeta$ on $\Ak$ with the property that all letters appear in the set of words $\{\zeta(a):\,a\in \Ak\}$ and there exists $a$ such that $|\zeta(a)|>1$.

Let $\ba^+ =(\zeta_n)_{n\ge 1}$ be a 1-sided sequence of substitutions on $\Ak$. Denote
$$
\zeta^{[n]} := \zeta_1\circ \cdots\circ\zeta_n,\ \ n\ge 1.
$$
Recall that $\Sf_{\zeta_1\circ \zeta_2} = \Sf_{\zeta_1}\Sf_{\zeta_2}$.
We will sometimes write
$$
\Sf_j:= \Sf_{\zeta_j}\ \ \mbox{and}\ \ \Sf^{[n]}:= \Sf_{\zeta^{[n]}}.
$$
We will also consider subwords of the sequence $\ba$ and the corresponding substitutions obtained by composition. Denote
\be \label{notation1}
\Sf_\bq = \Sf_n\cdots \Sf_\ell\ \ \mbox{and}\ \ A(\bq) = S_\bq^t\ \ \mbox{for}\ \ \bq = \zeta_{n}\ldots\zeta_{\ell}
\ee

Given $\ba^+$, denote by $X_{\ba^+}\subset \Ak^\Z$ the subspace of all two-sided sequence whose every subword appears as a subword of 
$\zeta^{[n]}(b)$ for some $b\in \Ak$ and $n\ge 1$. Let $T$ be the left shift on $\Ak^\Z$; then $(X_{\ba^+},T)$ is the (topological) $S$-adic dynamical system.
We refer to \cite{BD,BST_2019,BSTY} for the background on $S$-adic shifts.

There is a canonical correspondence between 1-sided Bratteli-Vershik diagrams with a Vershik's ordering, having $m$ vertices on each level, and sequences of substitutions $\ba^+=(\zeta_j)_{j\ge 1}$ on the alphabet 
$\Ak = \{0,\ldots,m-1\}$. For a stationary BV diagram (which corresponds to a single substitution), it was discovered by Livshits \cite{Liv1}. 
Given a Vershik's ordering $\ok$ on a BV diagram
$\{\Gam_{j}\}_{j\ge 1}$, the corresponding $S$-adic system is constructed as follows.  The
alphabet $\Ak = \{1,\ldots,m\}$ is identified with the vertex set of all the graphs $\Gam_n$.
The substitution $\zeta_j$ takes every $b\in \A$ into the word in $\A$ corresponding to all the vertices to which
there is a $\Gam_{j}$-edge starting at $b$, in the order determined by $\ok$. Formally, the {\em length} of the word $\zeta_j(b)$ is
$$
|\zeta_j(b)| = \sum_{a=1}^m A_{b,a}(\Gam_j),
$$
and the substitution itself is given by 
\be \label{def-subs1}
\zeta_j(b) =u_1^{b,j}\ldots u_{|\zeta_j(b)|}^{b,j},\ \ \ b\in \A,\ \ j\ge 0,
\ee
where $(b, u_i^{b,j})\in \Ek(\Gam_j)$, listed in the linear order prescribed by $\ok$. Note that the Standing Assumption on the sequence of graphs $\{\Gam_j\}$ implies the following:


\medskip

{\bf (A1)} {\em There  is a word $\bq\in \FrA^*$ which appears in $\ba^+$ infinitely often, for which $\Sf_\bq$ has all entries strictly positive.}
 
 \medskip
 
Property (A1) implies minimality and unique ergodicity of the $S$-adic shift, see \cite[Theorems 5.2 and 5.7]{BD};  the claim on unique ergodicity parallels that of Markov compactum and similarly goes back to Furstenberg \cite[(16.13)]{furst}.

We will also assume that the $S$-adic system is {\em aperiodic}, i.e., it has no periodic points. (A minimal system that has a periodic point, is a system on a finite space, and we want to exclude a trivial situation.) 

Further, we need the notion of  {\em recognizability} for the sequence of substitutions, introduced 
in \cite{BSTY}, which generalizes {\em bilateral recognizability} of B. Moss\'e \cite{Mosse} for a single substitution, see also Sections 5.5 and 5.6 in \cite{Queff}. By definition of the space $X_{\ba^+}$, for every $n\ge 1$, every $x\in X_{\ba^+}$ has a representation of the form
\be \label{recog}
x = T^k\bigl(\zeta^{[n]}(x')\bigr),\ \ \mbox{where}\ \ x'\in X_{\sig^n \ba^+},\ \ 0\le k < |\zeta^{[n]}(x_0)|.
\ee
Here $\sig$ denotes the left shift on $\FrA^\N$, and we recall that a substitution $\zeta$ acts on $\Ak^\Z$ by
$$
\zeta(\ldots a_{-1}.a_0 a_1\ldots) = \ldots \zeta(a_{-1}).\zeta(a_0)\zeta(a_1)\ldots
$$

\begin{defi} \label{def-recog}
A sequence of substitutions $\ba^+ = (\zeta_j)_{j\ge 1}$ is said to be {\em recognizable} if the representation (\ref{recog}) is unique for all $n\ge 1$.
\end{defi}

The following is a special case of \cite[Theorem 4.6]{BSTY} that we need.

\begin{theorem}[{\cite{BSTY}}] \label{th-recog0}
Let $\ba^+= (\zeta_j)_{j\ge 1} \in \FrA^\N$ be such that $\det(\Sf_{\zeta_j})\ne 0$ for every substitution matrix and $X_{\ba^+}$ is aperiodic. Then $\ba^+$ is recognizable.
\end{theorem}

\begin{theorem}[{\cite[Theorem 6.5]{BSTY}}] \label{thm-recog}
Let $\gf^+=\{\Gam_n\}_{n\ge 1}$ be a 1-sided BV-diagram equipped with a Vershik's ordering $\ok$, and let $\ba^+\in \FrA^\N$ be the corresponding sequence of substitutions.
If $\ba^+$ is recognizable, then the $S$-adic shift $(X_{\ba^+},T)$ is almost topologically conjugate, hence measurably conjugate in case the system is uniquely ergodic, to the 
Vershik automorphism on $X^+(\gf^+)$.
\end{theorem}


\subsection{Suspension flows and cylindrical functions} 
Condition (A1) implies that $(X_{\ba^+},T)$ is minimal and uniquely ergodic, and we
denote the unique invariant probability measure by $\mu_{\bom^+}$. 
 We further let
 $(\Xx_{\ba^+}^{\vec{s}}, h_t, \wt{\mu}_{\bom^+})$ be the suspension flow over $(X_{\ba^+},\mu_{\bom^+}, T)$,  corresponding to a piecewise-constant roof function $\phi$ defined by $\vec{s}\in \R^m_+$, that is,
$$
\phi(x) = s_{x_0},\ \ x\in X_{\ba^+}.
$$
The measure $\wt\mu_{\bom^+}$ is induced by the product of $\mu_{\bom^+}$ and the Lebesgue measure on $\R$.
 By definition, we have a union, disjoint in measure:
$$
\Xx_{\ba^+}^{\vec{s}} = \bigcup_{a\in \Ak} [a]\times [0,s_a],
$$
where $X_{\ba^+} = \bigsqcup_{a\in \Ak} [a] $ is the partition into cylinder sets according to the value of $x_0$. It is convenient to use the normalization:
$$
\vec{s} \in \Delta^{m-1}_\bom:= \Bigl\{\vec s\in \R^m_+: \ \sum_{a\in \Ak} \mu_{\bom^+}([a]) \cdot s_a = 1\Bigr\},
$$
so that $\wt \mu_{\bom^+}$ is a probability measure on $\Xx_{\ba^+}^{\vec{s}}$. Below we often omit the subscript and write $\mu = \mu_{\bom^+}$, $\wt\mu = \wt\mu_{\bom^+}$,
when it does not cause a confusion.

\medskip

Define {\em bounded cylindrical functions} (or {\em cylindrical functions of level zero}) by the formula: 
\be \label{fcyl2}
f(x,t)=\sum_{a\in \Ak} \One_{[a]}(x) \cdot \psi_a(t),\ \ \mbox{with}\ \ \psi_a\in L^\infty[0,s_a].
\ee
Cylindrical functions of level zero do not suffice to describe the spectral type of the flow; rather, we need functions depending on an arbitrary fixed number of symbols. 
We assume that the sequence of substitutions $\ba^+$ is recognizable, and say that $f$ is a  {\em bounded cylindrical function of level} $\ell$ if
\be \label{fcyl3}
f(x,t)=\sum_{a\in \Ak} \One_{\zeta^{[\ell]}[a]}(x) \cdot \psi^{(\ell)}_a(t),\ \ \mbox{with}\ \ \psi^{(\ell)}_a\in L^\infty [0,s^{(\ell)}_a],
\ee
where
$$
\vec{s}^{\,(\ell)}= (s^{(\ell)}_a)_{a\in \Ak}:= (\Sf^{[\ell]})^t \vec{s}.
$$
{This way of writing is convenient,  but does not stress explicit dependence on words in $\Ak^\ell$. In the case of IET, this representation corresponds to the $m$ intervals of the interval exchange obtained after $\ell$ steps of the Rauzy induction, where we do not renormalize the total length. Then the ``heights'' grow, as do the vectors $\vec s^{(\ell)}$ here, and the dependence of the function on the past from $1$ to $\ell-1$ is ``hidden'' in the functions $\psi^{(\ell)}_a$.}
The justification of the representation \eqref{fcyl3} requires recognizability of $\ba^+$, which implies that 
\be \label{KR}
\Pk_n = \{T^i(\zeta^{[n]}[a]):\ a\in \Ak,\ 0 \le i < |\zeta^{[n]}(a)|\}
\ee
is a sequence of Kakutani-Rokhlin partitions for $n\ge n_0(\ba)$, which generates the Borel $\sig$-algebra on the space $X_{\ba^+}$. We emphasize that, in general, $\zeta^{[n]}[a]$ may be a proper subset of $[\zeta^{[n]}(a)]$.

We have a union, disjoint in measure:
\be \label{ell-decom}
\Xx_{\ba^+}^{\vec{s}} = \bigcup_{a\in \Ak} \zeta^{[\ell]}[a]\times [0, s_a^{(\ell)}],
\ee
and so bounded cylindrical functions of level $\ell$ are well-defined by (\ref{fcyl3}).

{\
\subsection{Shift  for $S$-adic systems}

We next describe the relation between suspension flows over the recognizable uniquely ergodic $S$-adic system $(X_{\bom^+},T, \mu_{\bom^+})$ and the ``shifted'' one $(X_{\sig^\ell\bom^+}, T, \mu_{\sig^\ell\bom^+})$, which is analogous to the shift of a 2-sided Bratteli-Vershik diagram.}
Using the uniqueness of the representation (\ref{recog}) and the Kakutani-Rokhlin partitions (\ref{KR}), we obtain for $n\ge n_0$:
$$
\mu_{\bom^+}([a]) = \sum_{b\in \Ak}\Sf^{[n]}(a,b)\,\mu_{\bom^+}(\zeta^{[n]}[b]),\ \ a\in \Ak,
$$
hence
\be \label{eq-measure}
\vec{\mu}_0 = \Sf^{[n]}\vec{\mu}_n,\ \ \mbox{where}\ \ \vec{\mu}_n = \bigl(\mu_{\bom^+}(\zeta^{[n]}[b])\bigr)_{b\in \Ak}
\ee
is a column-vector. Observe also that 
\be \label{eq-measure2}
\vec{\mu}_n = \Sf_{n+1}\vec{\mu}_{n+1},\ \  n\ge n_0. 
\ee
It follows from the above that, for any $\ell\ge 1$, the suspension flow $(\Xx_{\ba^+}^{\vec{s}},\wt\mu_{\bom^+},h_t)$ is measurably isomorphic to the suspension flow over the system $(\zeta^{[\ell]}(X_{\sig^\ell\ba^+}), \zeta^{[\ell]}\circ T\circ (\zeta^{[\ell]})^{-1})$, with the induced measure, and a piecewise-constant roof function given by the vector
$
\vec{s}^{\,(\ell)}= (\Sf^{[\ell]})^t \vec{s}.
$
{\ Notice that $\langle \vec\mu_\ell, \vec s^{\ell}\rangle=1$. In the symbolic representation of translation flows, which we describe in detail below in Sections 2.9-2.11, this corresponds to the Rauzy-Veech induction, in which the intervals of the exchange get shorter and the ``roof'' higher. Since $(X_{\sig^\ell\bom^+}, \mu_{\sig^\ell\bom^+})$ is a probability space, we need to renormalize 
(to continue the analogy, to make the new base interval have unit length).
It is easy to see that 
$$
\bigl(\mu_{\sig^\ell \bom^+}([b])\bigr)_{b\in \Ak} = \frac{\vec\mu_\ell}{{\|\vec \mu_\ell\|}_1}.
$$
Thus $\vec s^{(\ell)}{\|\vec \mu_\ell\|}_1\in \Delta^{m-1}_{\sig^\ell\bom}$, and we obtain the following
\begin{lemma} \label{lem-shift}
The suspension flows $(\Xx_{\ba^+}^{\vec{s}},\wt{\mu}_{\bom^+},h_t)$ and $\bigl(\Xx_{\sig^\ell\ba^+}^{\vec{s}^{\ell}{\|\vec \mu_\ell\|}_1},\wt{\mu}_{\sig^\ell\bom^+},h_t\bigr)$ are  measurably isomorphic.
\end{lemma}

\subsection{Telescoping an $S$-adic sequence \cite[5.2]{BSTY}}
Given a sequence of substitutions $\ba^+ = (\zeta_n)_{n\ge 1}$ and an increasing sequence of integers $(n_k)_{k\ge 1}$ with $n_1=1$, we define the {\em telescoping} of $\ba^+$ along $(n_k)_{k\ge 1}$
to be the sequence of substitutions $\wt\ba^+ = (\wt \zeta_k)_{k\ge 1}$ where $\wt\zeta_k = \zeta_{n_k}\cdots \zeta_{n_{k+1}-1}$. It is immediate from the definitions that $X_{\ba^+}=X_{\wt\ba^+}$, so the
resulting dynamical systems are identical. This  is parallel to the operation of aggregation/telescoping for Bratteli-Vershik diagrams, mentioned in Section 2.1.
}


\subsection{Weakly Lipschitz functions.} \label{sec-weakl}
Following Bufetov \cite{Bufetov1, Buf-umn}, we consider the space of {\em weakly Lipschitz functions} on the space $\Xx_{\ba^+}^{\vec{s}}$, except here we do everything in the $S$-adic framework. This is the class of functions that we obtain from Lipschitz functions on the translation surface $M$ under the symbolic representation of the translation flow, for almost every Abelian 
differential.

\begin{defi}
Suppose that $\ba^+\in \FrA^\N$ is such that the $S$-adic system $(X_{\ba^+},T)$ is uniquely ergodic, with the invariant probability measure $\mu$. 
We say that a bounded function $f:\Xx_{\ba^+}^{\vec{s}}\to \C$ is weakly Lipschitz and write $f\in \Lip_w (\Xx_{\ba^+}^{\vec{s}})$
if there exists $C>0$ such that for all $a\in \Ak$ and $\ell\in \N$, for any $x,x'\in \zeta^{[\ell]}[a]$ and all $t\in [0, s_a^{(\ell)}]$, we have
\be \label{LL2}
|f(x,t)-f(x',t)|\le C \mu_{\ba^+}(\zeta^{[\ell]}[a]).
\ee
Here we are using the decomposition of $\Xx_{\ba^+}^{\vec{s}}$ from (\ref{ell-decom}).
The norm in $Lip_w  (\Xx_{\ba^+}^{\vec{s}})$ is defined by
\be \label{Lip-norm}
{\|f\|}_L:= {\|f\|}_\infty + \wtil{C},
\ee
where $\wtil{C}$ is the infimum of the constants in (\ref{LL2}). 
\end{defi}

Note that a  weakly Lipschitz functions is not assumed to be Lipschitz in the $t$-direction. This direction corresponds to the ``past'' in the 2-sided Markov compactum and to the vertical direction in the space of the suspension flow under the symbolic representation, and the reason is that any symbolic representation of a flow on a manifold unavoidably has discontinuities.

{\ Under the isomorphism of Lemma~\ref{lem-shift}, for any $\ell\ge 1$ a weakly Lipschitz functions on $\Xx_{\ba^+}^{\vec{s}}$ is mapped to a weakly Lipschitz function on
$\Xx_{\sig^\ell\ba^+}^{\vec{s}^{\ell}{\|\vec \mu_\ell\|}_1}$, and this map does not increase the norm $\|\cdot\|_L$. Similarly, telescoping the sequence $\ba^+$ does not increase the norm $\|\cdot\|_L$
of a weakly Lipschitz function.}

\begin{lemma} \label{lem-approx}
Let $f:\Xx_{\ba^+}^{\vec{s}}\to \C$  be a weakly Lipschitz function. Then for any $\ell\in \N$ there exists a bounded cylindrical function of level $\ell$, denote it $f^{(\ell)}$, such that
$\|f^{(\ell)}\|_\infty \le \|f\|_\infty$ and
$$
{\|f - f^{(\ell)}\|}_\infty \le  {\|f\|}_L\cdot \max_{a\in \Ak} \mu_{\bom^+}(\zeta^{[\ell]}[a]).
$$
\end{lemma}

\begin{proof}
We use the decomposition (\ref{ell-decom}). For each $a\in \Ak$ and $\ell$ choose $x_{a,\ell}\in \zeta^{[\ell]}$ arbitrarily, and let
$$
f^{(\ell)}(x,t):= f(x_{a,\ell},t),\ \ \mbox{where}\ \ x_0=a,\ t\in [0, s_a^{(\ell)}].
$$
By definition, the function $f^{(\ell)}$ has all the required properties.
\end{proof}


\subsection{Interval exchange transformations and suspensions over them} \label{subsec-IET}

We recall briefly the  connection between translation flows and {\em interval exchange transformations}, discovered by Veech \cite{Veech0,
veech}; for more details see, e.g., the surveys by Viana \cite{viana2}, Yoccoz \cite{yoccoz}, and Zorich \cite{Zorich1}. Let  $\pi\in \Frs_m$ be a permutation of $\{1,\ldots,m\}$ for $m\ge 2$.
An interval exchange transformation (IET) $f(\lam,\pi)$ is determined by $\pi$ and a positive vector $\lam = (\lam_1,\ldots,\lam_m)\in \R^m_+$. It is a piecewise isometry on the interval
$I = \bigcup_{j=1}^m I_j$, where $I_j = \Bigl[\sum_{k<j}\lam_k, \sum_{k\le j} \lam_j\Bigr)$ for $j=1,\ldots,m$, in which the intervals $I_j$ are translated and exchanged according to the permutation $\pi$.
To be precise,
$$
f(\lam,\pi):\ x\mapsto x + \sum_{\pi(j) < \pi(i)} \lam_j - \sum_{j < i} \lam_j,\ \ \ x\in I_i,
$$
which means that {\em after the exchange} the interval $I_j$ is in the $\pi(j)$-the place. Here we are using the convention of Veech; some authors use different notation.
We assume that $\pi$ is irreducible, i.e., $\pi\{1,\ldots,k\}\ne \{1,\ldots,k\}$ for $k<m$.

For $m=2$ the IET is just a circle rotation (modulo identification of the endpoints of $I$), and it can be viewed as the first return map of a linear flow on a torus $\T^2$. Similarly, for $m\ge 3$ by a singular suspension (with a piecewise-constant roof function, constant on each subinterval $I_i$), the IET can be represented as a first return map of a translation flow on a suitable translation surface to a specially chosen Poincar\'e section, a line segment $I$, see \cite{Veech0,veech}. Conversely, 
given a translation surface, one can find a horizontal segment $I$ in such a way that the first return map of the vertical flow to $I$ is an IET. Precise connection between the two systems is given by the {\em zippered rectangles construction} of Veech \cite{veech}.

\subsection{Rauzy-Veech-Zorich induction and the corresponding cocycles} \label{subsec-Rauzy}
A fundamental tool in the study of IET's and translation flows is the Rauzy-Veech induction, introduced   in \cite{Veech0,Rauzy}. 
Let $\pi\in \Frs_m$, and suppose that $(\lam,\pi)$ is such that $\lam_m \ne \lam_{\pi^{-1}(m)}$. Then the first return map of $f(\lam,\pi)$ to the interval
$
\bigl[0, \sum_{i=1}^m \lam_i - \min\{\lam_{\pi^{-1}(m)},\lam_m\} \bigr)
$
is an  IET on $m$ intervals $f(\lam',\pi')$ as well, see, e.g., \cite{viana2,yoccoz}. This defines a map
$$
\Qk_R:\ \R^m_+\times \Frs_m \to \R^m_+\times \Frs_m,\ \ \ \Qk_R(\lam,\pi) = (\lam',\pi'),
$$
on a set of full Lebesgue measure. Moreover, if $\pi$ is irreducible, then $\pi'$ is irreducible as well.
If $\lam_m < \lam_{\pi^{-1}(m)}$, we say that this is an operation of type ``a''; otherwise, an operation of type ``b''. The {\em Rauzy graph} is a directed labeled graph, whose vertices are  permutations of $\Ak=\{1,\ldots,m\}$ and the edges lead to permutations obtained by applying one of the operations. The edges are labeled ``a'' or ``b'' depending on the type of the operation. The {\em Rauzy class} of a permutation $\pi$ is the set of all permutations that can be reached from $\pi$ following a path in the Rauzy graph. 
For almost every IET (with respect to the Lebesgue measure on $\R^m_+$), the algorithm is well-defined for all times into the future, 
and we obtain an infinite path in the Rauzy graph, corresponding to the IET. Veech \cite{veech} proved that, conversely, every infinite path in the Rauzy graph arises from an IET in a such a way.

In the ergodic theory of IET's it is useful to consider an {\em acceleration} of the algorithm. Zorich induction \cite{Zorich,Zorich1} is obtained by applying the Rauzy-Veech induction  until the first switch from a type ``a'' to a type ``b'' operation, or vice versa.
Sometimes other versions of the algorithm and accelerations are used, see, e.g., Marmi, Moussa, and Yoccoz \cite{MMY}.

Let $(\lam',\pi')=\Qk_R(\lam,\pi)$, and denote $f_I=f(\lam,\pi)$. Write $I_j = I_j(\lam,\pi)$ and let $J_j = I_j(\lam',\pi')$ be the intervals of the exchange $f_J=f(\lam',\pi')$. Denote by 
$r_i$ the return time for the interval $J_i$ into $J$ under $f_I$, that is,
$r_i = \min\{k>0: f_I^k (J_i) \subset J\}$. From the definition of the induction procedure it follows that $r_i = 1$ for all $i$ except one, for which it is equal to 2. 
Represent $I$ as a Rokhlin tower over the subinterval $J$ and its induced map $f_J$, and write
\be \label{Rokhlin1}
I = \bigsqcup_{i=1,\ldots,m,\ k = 0, \ldots, r_i-1} f_I^k (J_i).
\ee
By construction, each of the ``floors'' of our tower, that is, each of the subintervals $f_I^k (J_i)$ is a subset of a unique subinterval of the initial exchange, and we define an integer $n(i,k)$ by the formula
\be \label{Rokhlin2}
f_I^j (J_i) \subset I_{n(i,k)}.
\ee
Let $B_R(\lam,\pi)$ be the linear operator on $\R^m$ given by the $m\times m$ matrix $[n(i,j)]$. This matrix is easily shown to be unimodular. Given a Rauzy class $\FrR$, the function $B_R: \R^m_+\times \FrR \to GL(m,\R)$ yields the {\em Rauzy-Veech}, or {\em renormalization} cocycle. If, instead, we apply the Zorich induction algorithm, the same procedure yields the {\em Zorich cocycle}. 

One can consider the Rauzy-Veech and Zorich  induction algorithm also on the set of zippered rectangles; these can be represented as bi-infinite paths in the Rauzy graph. After an appropriate renormalization, the Rauzy-Veech map $(\lam,\pi)\mapsto (\lam',\pi')$ and the Zorich map $(\lam,\pi) \mapsto (\lam'',\pi'')$ can be seen as the first return maps of the Teichm\"uller flow on the space of zippered rectangles, with respect to carefully chosen Poincar\'e sections, see, e.g., \cite[Section 2.10]{viana2} and \cite[Section 11.3]{yoccoz}.

\medskip

\indent
{\bf Remark.} The zippered rectangles construction provides natural bases for the absolute and relative homology groups $H_1(M\setminus \Sig,\R)$ and $H_1(M,\Sig,\R)$; in particular, $\R^m$ may be identified with  $H_1(M,\Sig,\R)$. The Rauzy-Veech cocycle can then be represented as acting on the cohomology groups 
$H^1(M\setminus \Sig,\R)$ and $H^1(M,\Sig,\R)$, as shown by Veech \cite{veech} (see also \cite[Section 2.9]{viana2}).


\subsection{Symbolic representation of IET's and translation flows} \label{subsec-symbol}

Let $\Hk$ be a connected component of a stratum of genus $g\ge 2$ and $\FrR$ the Rauzy class of a permutation $\pi\in \Frs_m$ corresponding to $\Hk$. It is known that $m\ge 2g$.
Veech \cite{veech}
constructed a measurable  map from the space $\Vk(\FrR)$ of
zippered rectangles corresponding to the Rauzy class $\FrR$, to $\Hk$, which intertwines the Teichm\"uller flow on $\Hk$ and a renormalization flow $P_t$ that Veech defined on $\Vk(\FrR)$. This  translation flow on the flat surface is measurably isomorphic to the suspension flow over an IET in the Rauzy class $\FrR$. The roof function of the suspension is constant on each interval of the exchange and can therefore be expressed as a vector of ``heights''.
It was shown by Veech \cite{veechamj} that the ``vector of heights'' obtained in this construction necessarily belongs to a subspace $H(\pi)$, which is invariant under the Rauzy-Veech cocycle and depends only on the permutation of the IET. In fact, the subspace $H(\pi)$ has dimension $2g$ and is the sum of the stable and unstable subspaces for the Rauzy-Veech cocycle.

Section 4.3 of \cite{Buf-umn}  gives a symbolic coding of the flow $P_t$ on $\Vk(\FrR)$, namely, a map
\be \label{map-veech}
\Zk_{\FrR}: (\Vk(\FrR),\wt\nu) \to (\Om,\P)
\ee
defined almost everywhere, where $\wt\nu$ is the pull-back of  $\nu$ from $\Hk$, an invariant and ergodic map under the Teichm\"uller flow on
$\Hk$ and $(\Om,\P)$ is the space of Markov compacta with a probability measure $\P$.
The first return map of the flow $P_t$ for an appropriate Poincar\'e section is mapped by $\Zk_\FrR$ to the shift map $\sigma$ on $(\Om,\P)$. This correspondence maps the Rauzy-Veech cocycle over the Teichm\"uller flow into the renormalization cocycle for the Markov compacta. Moreover, the Markov compacta in the image of $\Zk_{\FrR}$ are equipped with a
canonical Vershik's ordering.
The transformation $\Zk_\Rk$ induces a map defined for a.e.\ $\Xk\in \Vk(\Rk)$, from the corresponding Riemann surface $M(\Xk)$ to a 2-sided Markov compactum $X$ with the Vershik's ordering, intertwining their vertical and horizontal flows. 

The vertical flow on $X$ is, in turn, measurably isomorphic to the suspension flow over the one-sided Markov compactum $X_+$. 
The Vershik automorphism on $X_+$ provides a symbolic representation of the IET from $\Rk$ on $m\ge 2g$ symbols.

{ In brief, the map $\Zk_\Rk$ is defined as follows. Identify an element of $\Vk(\FrR)$ with a suspension over an IET $(\lam,\pi)$.
We then run the 2-sided Rauzy-Veech induction, equivalently, consider a bi-infinite path in the Rauzy diagram, which is well-defined a.s., and take
$\Gam_n$ to be the graph whose incidence matrix is the matrix of the linear operator $B_R(\Qk^{n-1}_R(\lam,\pi))$ in the standard basis, see (\ref{Rokhlin1}), (\ref{Rokhlin2}). 

In this paper we are going to use the framework of $S$-adic systems. The justification for transition from the Bratteli-Vershik coding of \cite{Buf-umn} to the $S$-adic coding is provided by Theorems~\ref{thm-recog} and \ref{th-recog0}, in view of the fact that the matrices of the Rauzy-Veech cocycle are unimodular, see \cite{veech,veechamj}. In addition, note that the $S$-adic system is aperiodic, since the number of admissible ``words'' of length $n$ in the Rauzy graph is unbounded as $n\to \infty$.
The substitution $\zeta_1$ in the resulting symbolic representation can be ``read off'' the Rokhlin tower (\ref{Rokhlin1}), (\ref{Rokhlin2}) of one step of the Rauzy-Veech induction:
\be \label{induc}
\zeta_1:\ i \mapsto n(i,0)\ldots n(i, r_1-1),\ \ i=1,\ldots,m.
\ee
Thus we obtain
$$B_R(\lam,\pi) = [n(i,j)]_{i,j=1}^m = \Sf_{\zeta_1}^t.$$
Note that he property (A1) of the resulting sequence of substitutions
holds almost surely, see Veech \cite{veech}.}


\subsection{Spectral measures and twisted Birkhoff integrals} \label{sec-twist}
We use the following convention for the Fourier transform of functions and measures: given $\psi\in L^1(\R)$ we set
$
\widehat{\psi}(t) = \int_\R e^{-2\pi i \om t} \psi(\om)\,d\om,
$
and for a  probability measure $\nu$ on $\R$ we let
$
\widehat{\nu}(t) = \int_\R e^{-2\pi i \om t}\,d\nu(\om).
$

Given a measure-preserving flow
$(Y, h_t,\mu)_{t\in \R}$ and a test function $f\in L^2(Y,\mu)$, 
there is a finite positive Borel measure $\sig_f$ on $\R$ such that
$$
\widehat{\sig}_f(-\tau)=\int_{-\infty}^\infty e^{2 \pi i\om \tau}\,d\sig_f(\om) = \langle f\circ h_\tau, f\rangle\ \ \ \mbox{for}\ \tau\in \R.
$$
In order to obtain local bounds on the spectral measure, we use growth estimates of the twisted Birkhoff integral
\be \label{twist1}
S_R^{(y)}(f,\om) := \int_0^R e^{-2\pi i \om \tau} f\circ h_\tau(y)\,d\tau.
\ee
The following lemma is standard; a proof may be found in \cite[Lemma 4.3]{BuSol2}.

\begin{lemma} \label{lem-varr} Suppose that for some fixed $\om \in \R$, $R_0>0$, and $\alpha \in (0,1)$ we have
\be \label{L2est}
\left\|S_R^{(y)}(f,\om)\right\|_{L^2(Y,\mu)}\le C_1R^\alpha\ \ \mbox{for all}\ R\ge R_0.
\ee
Then 
\be \label{locest}
\sig_f([\om-r,\om+r]) \le \pi^2 2^{-2\alpha} C_1^2 r^{2(1-\alpha)}\ \ \mbox{for all}\ r \le (2R_0)^{-1}.
\ee
\end{lemma}


\section{Random $S$-adic systems: statement of the theorem}

Here we consider dynamical systems generated by a {\em random} $S$-adic system. In order to state our result, we need some preparation; specifically, the Oseledets Theorem.

Recall that $\FrA$ denotes the set of substitutions $\zeta$ on $\Ak$ with the property that all letters appear in the set of words $\{\zeta(a):\,a\in \Ak\}$ and there exists $a$ such that $|\zeta(a)|>1$.
Let $\Om$ be the 2-sided space of sequences of substitutions:
$$
\Om = \{\bom = \ldots \zeta_{-n}\ldots\zeta_0{\mbox{\bf .}}\zeta_1\ldots \zeta_n\ldots;\ \zeta_i \in \FrA,\ i\in \Z\},
$$
equipped with the left shift $\sig$.
For $\bom\in\Om$ we denote by $X_{\bom^+}$  the $S$-adic system corresponding to $\bom^+=\{\zeta_n\}_{n\ge 1}$.
 
We will sometimes write $\zeta(q)$ for the substitution corresponding to $q\in \FrA$. For a word $\bq = q_1\ldots q_k\in \FrA^k$ we can compose the substitutions to obtain
 $\zeta(\bq) = \zeta(q_1)\ldots \zeta(q_k)$.
 We will also need a ``2-sided cylinder set'':
$$
[\bq.\bq] = \{\bom\in \Om:\ \zeta_{-k+1}\ldots\zeta_0 = \zeta_1\ldots\zeta_k = \bq\}.
$$
Following \cite{BufGur}, we say that the word $\bq = q_1\ldots q_k$ is ``simple'' if for all $2 \le i \le k$ we have $q_i\ldots q_k \ne q_1\ldots q_{k-i+1}$. If the word $\bq$ is simple, two occurrences of $\bq$ in the sequence
$\bom$ cannot overlap. 

\begin{defi}
A word $v\in \Ak^*$ is a {\em return word} for a substitution $\zeta$ if $v$ starts with some letter $c$ and $vc$ occurs in the substitution space $X_\zeta$. The return word is called ``good'' if $vc$ occurs in the substitution $\zeta(j)$ of every letter. We denote by $GR(\zeta)$ the set of good return words for $\zeta$.
\end{defi}

Recall that
$
\AA(\bom) := \Sf^t_{\zeta_1}.
$
Let $\P$ be an ergodic $\sig$-invariant probability measure on $\Om$ satisfying the following 

\medskip

\noindent
{\bf Conditions:} 

{\bf (C1)} The matrices $\AA(\bom)$ are almost surely invertible with respect to $\P$.

{\bf (C2)} The functions $\bom\mapsto \log(1+ \|A^{\pm 1}(\bom)\|)$ are integrable. 

\noindent We can use any matrix norm, but it will be convenient the $\ell^\infty$ operator norm, so 
$\|A\|=\|A\|_\infty$ unless otherwise stated.

\medskip

We obtain a measurable cocycle $\AA: \Om\to GL(m,\R)$, called the {\em renormalization cocycle}. Denote
\begin{equation}\label{renormcoc}
\AA(n,\bom) = \left\{ \begin{array}{lr} \AA(\sig^{n-1}\bom)\cdot \ldots \cdot \AA(\bom), & n> 0; \\
                                                     Id, & n=0; \\
                                                      \AA^{-1}(\sig^{-n}\bom) \cdot \ldots \cdot \AA^{-1}(\sig^{-1}\bom), & n<0,\end{array} \right.
\end{equation}
so that 
$$\AA(n,\bom) = \Sf^t(\bom_n\ldots \bom_1)= \Sf^t_{\zeta_1\ldots\zeta_n},\ n\ge 1.$$
By the Oseledets Theorem \cite{oseledets} (for a detailed survey and refinements, see Barreira-Pesin \cite[Theorem 3.5.5]{barpes}),
there exist Lyapunov exponents $\theta_1> \theta_2 > \ldots > \theta_r$ and, for $\P$-a.e.\ $\bom\in \Om$, a direct-sum decomposition
\be \label{os1}
\R^m = E^1_\bom \oplus \cdots \oplus E^r_\bom 
\ee
that depends measurably on $\bom\in \Om$ and satisfies the following:

(i) for $\P$-a.e.\ $\bom\in \Om$, any $n\in\Z$, and any $i=1,\ldots,r$ we have
$$
\AA(n,\bom) E^i_\bom = E^i_{\sig^n\bom};
$$

(ii) for any $v\in E^i_\bom,\ v\ne 0$, we have
$$
\lim_{|n|\to \infty} \frac{\log\|\AA(n,\bom)v\|}{n} = \theta_i,
$$
where the convergence is uniform on the unit sphere $\{u\in E^i_\bom:\, \|u\|=1\}$;

\smallskip

(iii) $\lim_{|n|\to \infty} \frac{1}{n}\log\angle\left(\bigoplus_{i\in I} E^i_{\sig^n\bom}, \bigoplus_{j\in J} E^j_{\sig^n \bom}\right) =0$ whenever $I\cap J = \es$.

\medskip

We will denote by $E_\bom^{u}=\bigoplus\{E^i_\bom:\theta_i>0\}$ and $E_\bom^{st}=\bigoplus\{E^i_\bom:\theta_i<0\}$ respectively  the strong unstable and stable subspaces corresponding to $\bom$. Any subspace of the form $E_\bom^J:=\bigoplus_{j\in J} E^j_\bom$ will be called an Oseledets subspace corresponding to $\bom$.

\medskip

Let $\sig_f$ be the spectral measure for the system $(\Xx_{\bom^+}^{\vec{s}},h_t,\wt\mu_{\bom^+})$ with the test function $f$ (assuming the system is uniquely ergodic). 
{\ We will use the following notation: for a word $v$ in the alphabet $\Ak$ denote by $\vec{\ell}(v)\in \Z^m$ the positive vector whose $j$-th entry is the number of $j$'s in $v$, for $j\le m$,
and call it the  {\em population vector} of $v$.}
Now we can state our theorem.

\begin{theorem} \label{th-main1}
Let $(\Om,\P,\sig)$ be an invertible ergodic measure-preserving system satisfying conditions {\bf (C1)-(C2)} above. Consider the cocycle $\AA(n,\bom)$ defined by (\ref{renormcoc}).  Assume that
\begin{enumerate}
\item[(a)] 
there are $\kappa\ge 2$ positive Lyapunov exponents
and the top exponent is simple;
\item[(b)] there exists a simple {\ remove ``admissible''} word $\bq\in \FrA^k$ for some $k\in \N$, such that all the entries of the matrix $\Sf_\bq$ are strictly positive and $\P(\bqcyl)>0$;
\item[(c)]
the set of vectors $\{\vec\ell(v):\ v\ \mbox{is a good
return word for $\zeta$}\}$ generates $\Z^m$ as a free Abelian group;
\item[(d)]
Let $\ell_\bq(\bom)$ be the ``negative'' waiting time until the first appearance of $\bq.\bq$, i.e.
$$
\ell_\bq(\bom) = \min\{n\ge 1:\, \sig^{-n}\bom \in \bqcyl\}.
$$
Let $\P(\bom|\bom^+)$ be the conditional distribution on the set of $\bom$'s
conditioned on the future $\bom^+ = \bom_1\bom_2\ldots$ We assume that there exist $\eps>0$ and $1<C<\infty$ such that
\be \label{eq-star1}
\int_{\bqcyl} \left\|\AA(\ell_\bq(\bom), \sig^{-\ell_\bq(\bom)}\bom)\right\|^\eps\,d\P(\bom|\bom^+) \le C\ \ \ \mbox{for all}\ \ \bom^+\  \mbox{starting with}\ \q.
\ee
\end{enumerate}
Then there exists $\gam>0$  such that for $\P$-a.e.\ $\bom \in \Om$ the following holds:

Let $(X_{\bom^+},T,\mu_{\bom^+})$ be 
the $S$-adic system corresponding to $\bom^+$, which is uniquely ergodic.
 Let $(\Xx_{\bom^+}^{\vec{s}},h_t,\wt \mu_{\bom^+})$ be the suspension flow over $(X_{\bom^+},T,\mu_{\bom^+})$ under the piecewise-constant roof function determined by $\vec{s}$.
Let $H^J_\bom$ be an Oseledets subspace corresponding to $\bom$, such that $E^u_\bom\subset H^J_\bom$. Then 
for Lebesgue-a.e.\ $\vec{s}\in H^J_\bom\cap \Delta_\bom^{m-1}$,
for all
   $B>1$ there exist $R_0 = R_0(\bom,\vec s, B)>1$ and $r_0=r_0(\bom, \vec s, B)>0$
 such that for any $f\in \Lip_w^+(\Xx_{\bom^+}^{\vec{s}})$, 
\be \label{int-growth1}
|S^{(x,t)}_R(f,\om)| \le \wtil{C}(\bom,\|f\|_L)\cdot R^{1-\gam/2},\ \ \mbox{for all}\ \om\in[B^{-1},B]\ \mbox{and}\ R\ge R_0,
\ee
uniformly in $(x,t)\in \Xx_{\bom^+}^{\vec{s}}$, and 
\be \label{main-Hoeld}
\sig_f(B(\om,r))\le C(\bom,\|f\|_L)\cdot r^\gam,\ \ \ \mbox{for all}\ \ \om\in[B^{-1},B]\ \mbox{and}\ 0<r< r_0,
\ee
with the constants depending only on $\bom$ and $\|f\|_L$.

More precisely, for any $\epsilon_1>0$ there exists $\gam(\epsilon_1)>0$, such that for $\P$-a.e.\ $\bom \in \Om$  there is an exceptional set $\Fre(\bom,\epsilon_1)\subset H_\bom^J\cap \Delta_\bom^{m-1}$, satisfying
\be \label{dimesta}
\dim_H(\Fre(\bom,\epsilon_1)) < \dim(H^J_\bom)-\kappa+ \epsilon_1,
\ee
such that (\ref{int-growth1}) and (\ref{main-Hoeld}) hold for all $\vec s\in H^J_\bom \cap \Delta_\bom^{m-1}\setminus \Fre(\bom,\eps_1)$ with $\gam = \gam(\epsilon_1)$.
\end{theorem} 

\noindent {\bf Remarks.}
1. Note that $\dim( H^J_\bom \cap \Delta_\bom^{m-1}) = \dim(H^J_\bom)-1$ and $\kappa\ge 2$, so (\ref{dimesta}) indeed implies that $\Fre(\bom,\epsilon_1)$ has zero Lebesgue measure in $H^J_\bom \cap \Delta_\bom^{m-1}$.
 
2. The assumption that $\bq$ is a simple word ensures that occurrences of $\bq$ do not overlap. Then we have, in view of (\ref{notation1}):
\be \label{egstar}
\AA(\ell_\bq(\bom), \sig^{-\ell_\bq(\bom)}\bom) = A(\bq) A(\bp) A(\bq),
\ee
for some $\bp \in \FrA$ (possibly trivial). For our application, it will be easy to make sure that $\bq$ is simple, as we show in Section~\ref{sec-deriv}, unlike in the paper \cite{BufGur}, where additional efforts were needed to achieve the desired aims.

{ 3. We need to work in the Oseledets subspace $H^J_\ba$, rather than the entire space $\R^m$, in order to handle the case of the strata with
$m>2g$. Indeed, in order to make a claim for a.e.\ translation flow in such a stratum, it is not sufficient to that it holds for the suspension with a.e.\ vector of heights in $\R^m$ over a.e.\ IET in the corresponding Rauzy class. Rather, we need such a claim to hold for a.e.\ vector of heights in the equivariant subspace $H(\pi)$ of Veech, which has dimension $2g$.}


\medskip

Before starting with the proof, we include  a mini-dictionary, translating between the geometric and symbolic language used in this paper.

\medskip

\begin{center}
\begin{tabular}{|p{7.2cm}|p{7.2cm}|}
\hline
\begin{center}\textbf{Symbolic language}\end{center}&\begin{center}\textbf{Geometric language}\end{center}\\
\hline
$\Om$ -- 2-sided space of substitutions (space of random 2-sided Markov compacta), with a $\sig$-invariant ergodic measure $\P$ & Moduli space of abelian differentials of genus $\ge 2$, with an ergodic measure $\nu$, invariant under the Teichm\"uller flow \\
\hline
$\ba\in \Om$, 2-sided sequence of substitutions (Markov compactum) & Abelian differential \\
\hline
S-adic system $(X_{\ba^+},T)$ (Bratteli-Vershik automorphism) &Interval Exchange Transformation (IET)\\
\hline
Shift $\sig$ on $\Om$ & Rauzy-Veech-Zorich induction \\
\hline
Substitution (adjacency) matrices  & Rauzy-Veech Matrices \\
\hline
Suspension flow over $(X_{\ba^+},T)$ & translation flow (suspension over IET) \\
\hline
\end{tabular}
\end{center}

\bigskip



\section{Beginning of the proof of Theorem~\ref{th-main11} }
\subsection{Reduction to an induced system}

Here we show that Theorem~\ref{th-main1} reduces to the case when 
\be \label{induce}
\bom_n = \bq \bp_n \bq\ \ \ \mbox{for all}\ \ n\in \Z,
\ee
where $\bq$ is a fixed word in $\FrA^k$ for some $k\in \N$, such that its incidence matrix  is strictly positive, and $\bp_n$ is arbitrary. In the next theorem we use the same notation as in Theorem~\ref{th-main1}.

\begin{theorem} \label{th-main11}
Let $\bq$ be a fixed simple word in $\FrA^k$ for some $k\in \N$, such that the incidence matrix of the substitution $\zeta = \zeta(\bq)$ is strictly positive.
Let $(\Om_\bq,\P,\sig)$ be an invertible ergodic measure-preserving system, as in the previous section, satisfying conditions {\bf (C1), (C2)}, and in addition, every symbol-substitution in the sequences $\ba=(\ba_n)_{n\in \Z} \in \Om_\bq$ can be written in the form (\ref{induce}).
Consider the cocycle $\AA(n,\bom)$ defined by (\ref{renormcoc}).  Assume that
\begin{enumerate}
\item[(a$'$)] there are $\kappa\ge 2$ positive Lyapunov exponents,
and the top exponent is simple;
\item[(c$'$)] there exist ``good return words'' $\{u_j\}_{j=1}^k$ for $\zeta=\zeta(\bq)$, such that $\{\vec{\ell}(u_j)\}_{j=1}^k$ generates $\Z^m$ as a free Abelian group;
\item[(d$'$)]  there exist $\eps>0$ and $1<C<\infty$ such that
\be \label{eq-star11}
\int_{\Om_\bq} \left\|A(\bom_0)\right\|^\eps\,d\P(\bom|\bom^+) \le C\ \ \ \mbox{for all}\ \ \bom^+.
\ee
\end{enumerate}
Then the same conclusions hold as in Theorem~\ref{th-main1}.
\end{theorem} 

The properties (a$'$), (c$'$), and (d$'$) are analogues of (a), (c), (d) from Theorem~\ref{th-main1}. The analogue of the property (b) in Theorem~\ref{th-main1} holds automatically.

 
\begin{proof}[Proof of Theorem~\ref{th-main1} assuming Theorem~\ref{th-main11}] Given an ergodic system $(\Om,\P,\sig)$ from the statement of Theorem~\ref{th-main1}, we consider the induced system on the cylinder set
$\Om_\bq:= [\bq.\bq]$. 
{ Since $\bq$ is a simple word, the occurrences of the word $\bq.\bq$ in $\bom\in \Om$ are non-overlapping, and so 
 we can represent elements of $\Om_\bq$ symbolically as sequences satisfying (\ref{induce}). Denote by $\P_{\!\!\bq}$ the induced (conditional) measure on $\Om_\bq$. Since $\P([\bq.\bq])>0$, standard results in Ergodic Theory imply that the resulting induced system $(\Om_\bq,\P_{\!\!\bq},\sig)$ is also ergodic and the induced cocycle has the same properties of the Lyapunov spectrum (with the values of the Lyapunov exponents multiplied  by $1/\P(\bqcyl)$); that is, (a$'$) holds. The property (c$'$) follows from (c) automatically. Finally, note that (\ref{eq-star11}) is identical to (\ref{eq-star1}), and
 so Theorem~\ref{th-main11} may be applied.
 
 Now, for a $\P$-typical $\ba\in \Om$, let $\ell\ge 0$ be minimal such that $\sig^\ell\ba \in [\bq,\bq]$. Concatenating the symbols of $\sig^\ell\ba$ from one occurrence of $[\bq,\bq]$ to the next, we
obtain a $\P_{\!\!\bq}$-typical point $\bw\in \Om_\bq$.
For $\P$-a.e. $\ba\in \Om$, from the Oseledets bundle $H_\ba^J$ given in Theorem~\ref{th-main1} we get an induced Oseledets bundle $H_\bw^J$ with the property $E^u_\bw \subset H_\bw^J$.

 By 
 Theorem~\ref{th-main11}, we have for some $\gam>0$ the uniform bound on the twisted Birkhoff integral \eqref{int-growth1} and the H\"older property f\eqref{main-Hoeld} 
 for the spectral measure of an arbitrary
 weakly Lipschitz function on $\Xx_{\bw^+}^{\vec u}$, for Lebesgue-a.e.\ $\vec u\in H_\bw^J\cap \Delta_\bw^{m-1}$, with the spectral measure corresponding to the suspension flow
 $(\Xx_{\bw^+}^{\vec u}, h_t,\wt\mu_{\bw^+})$. Note that $\bw^+$ is obtained from $\sig^\ell\ba^+$ by a telescoping procedure, hence the corresponding $S$-adic spaces and suspension flows over them are naturally isomorphic. Further, $(\Xx_{\ba^+}^{\vec{s}},\wt{\mu}_{\bom^+},h_t)$ and $\bigl(\Xx_{\sig^\ell\ba^+}^{\vec{s}^{\ell}{\|\vec \mu_\ell\|}_1},\wt{\mu}_{\sig^\ell\bom^+},h_t\bigr)$ are isomorphic
 by Lemma~\ref{lem-shift}, and a.e.\ $\vec s\in H^J_\ba\cap \Delta_\ba^{m-1}$ gets mapped into a.e.
 $$
 \vec s^{(\ell)}\|\vec\mu_\ell\|_1 = (\Sf^{[\ell]})^t \vec s \cdot\|\vec\mu_\ell\|_1 = \AA(\ell,\ba)\vec s\cdot \|\vec\mu_\ell\|_1\in H_{\sig^\ell\ba}^J\cap \Delta^{m-1}_{\sig^\ell\ba}\cong H_\bw^J\cap \Delta_\bw^{m-1}.
 $$
 Note also that a weakly Lipschitz function on $\Xx_{\ba^+}^{\vec{s}}$ yields a weakly Lipschitz function on $\Xx_{\bw^+}^{\vec u}$, without increase of the norm $\|\cdot\|_L$. Thus the conclusions of Theorem~\ref{th-main11} yield the desired conclusions of Theorem~\ref{th-main1}, and the 
 reduction is complete.
 }
\end{proof}



\subsection{Exponential tails}

For $\bom\in  \Om_\bq$ we consider the sequence of substitutions $\zeta(\bom_n)$, $n\ge 1$. In view of (\ref{induce}), we have
$$
\zeta(\bom_n) =  \zeta(\bq) \zeta(\bp_n)\zeta(\bq).
$$
Recall that 
$$
A(\bp_n) = \Sf_{\zeta(\bp_n)}^t.
$$
Denote
\be \label{def-W}
W_n = W_n(\bom):=\log\|A(\bom_n)\| = \log\|Q^t A(\bp_n) Q^t\|.
\ee
It is clear that $W_n\ge 0$ for $n\ge 1$.

\begin{prop}[Prop.\,6.1 in \cite{BuSo18a}] \label{prop-proba}
Under the assumptions of Theorem~\ref{th-main11},
there exists a positive constant $L_1$  such that for $\P$-a.e.\ $\bom$, the following holds: for any $\delta>0$, for all $N$ sufficiently large ($N\ge N_0(\bom,\delta)$),
\be \label{w-cond2}
\max\left\{\sum_{n\in \Psi} W_{n+1}:\ \Psi\subset \ \{1,\ldots,N\},\  |\Psi| \le \delta N\right\} \le L_1\cdot \log(1/\delta)
\cdot \delta N.
\ee
\end{prop}

The following is an immediate consequence.

\begin{corollary} \label{cor-immed}
In the setting of Proposition~\ref{prop-proba}, we have for any $C\ge 1$:
\be \label{eq-immed}
\card\left\{n\le N:\ W_{n+1} > C L_1\cdot \log(1/\delta)\right\} \le \frac{\delta N}{C}\,.
\ee
\end{corollary}


\subsection{Estimating twisted Birkhoff integrals}

We will use the following notation: $\|y\|_{\R/\Z}$ is the distance from $y\in \R$ to the nearest integer.
We will  need the
``tiling length'' of $v$ defined, for $\vec s\in \R^m_+$, by 
\be \label{tilength}
|v|_{\vec{s}}:= \langle\vec{\ell}(v), \vec{s}\rangle,
\ee
{\ where $\vec{\ell}(v)$ is the population vector of $v$}.
Below we denote by $O_{\bom,Q}(1), O_{\bom,B}(1)$, etc.,  generic constants which depends only on the parameters indicated and which may be different from line to line. Recall that the roof function is normalized by
 $$
\vec{s}\in \Delta_{\bom}^{m-1}:= \{\vec{s}\in \R^m_+:\ \sum_{a\in \Ak} \mu_{\bom^+}([a]) s_a =1\}.
$$

\begin{prop} \label{prop-Dioph4} Suppose that the conditions of Theorem~\ref{th-main11} are satisfied. Then
for $\P$-a.e.\ $\bom\in \Om$, 
 for any $\eta\in (0,1)$, there exists $\ell_\eta=\ell_\eta(\bom)\in \N$, such that 
 for all $\ell\ge \ell_\eta$ and any bounded cylindrical function $f^{(\ell)}$ of level $\ell$, 
 for any $(x,t)\in \Xx_{\bom^+}^{\vec{s}}$,  $\vec s\in \Delta_\bom^{m-1}$,  and $\om\in \R$, 
\be \label{eq-Dioph3}
|S^{(x,t)}_R(f^{(\ell)},\om)| \le  O_{\bom,Q}(1)\cdot \|f^{(\ell)}\|_{_\infty} \Bigl(R^{1/2}+ R^{1+\eta}\!\!\!\!\!\prod_{\ell+1\le n < \frac{\log R}{4\theta_1}} 
\Bigl( 1 - c_1\cdot \!\!\!\!   \max_{v\in GR(\zeta)}\bigl\| \om|\zeta^{[n]}(v)|_{\vec{s}}\bigr\|^2_{\R/\Z}\Bigr)\Bigr),
\ee
for all
$
R\ge e^{8\theta_1 \ell}.
$
Here $c_1\in (0,1)$ is a constant depending only on $Q$.
\end{prop}

The proposition was proved in \cite[Prop. 7.1]{BuSo18a}, in the equivalent setting of Bratteli-Vershik transformations, and we do not repeat it here. The proof proceeds in several steps, which already appeared in one way or the other, in our previous work \cite{BuSol2,BuSo18a,BuSo19}. 
In short, given a cylindrical function of the form (\ref{fcyl2}), it suffices to consider $f(x,t)=\One_{[a]}(x) \cdot \psi_a(t)$ for $a\in \Ak$. A calculation shows that its twisted Birkhoff sum, up to a small error, equals $\widehat{\psi}_a$ times an exponential sum corresponding to appearances of $a$ in an initial word of a sequence $x\in X_{\bom^+}$. Using the prefix-suffix decomposition of $S$-adic sequences, the latter may be reduced to estimating exponential sums corresponding to the substituted symbols $\zeta^{[n]}(b)$, $b\in \Ak$. These together (over all $a$ and $b$ in $\Ak$) form a matrix of trigonometric polynomials to which we give the name of  a matrix Riesz product in
\cite[Section 3.2]{BuSo18a} and whose cocycle structure is studied  in \cite{BuSo19}. The next step is estimating the norm of a matrix product from above by the absolute value of a scalar product, which was done in \cite[Prop.\,3.4]{BuSo18a}. Passing from cylindrical functions of level zero to those of level $\ell$ follows by a simple shifting of indices, see \cite[Section 3.5]{BuSo18a}.The term $R^{1/2}$ (which can be replaced by any positive power of $R$ at the cost of a change in the range of $n$ in the product) absorbs several error terms.

One tiny difference with  \cite[Prop. 7.1]{BuSo18a} is that there we assumed a different normalization: ${\|s\|}_1=1$, hence ${\|s\|}_\infty\le 1$, which was used in the proof.  Here we we have
${\|s\|}_\infty \le \bigl(\min_{a\in \Ak} \mu_{\bom^+}([a])\bigr)^{-1}$, which is absorbed into the constant $O_{\bom,Q}(1)$.


\subsection{Reduction to the case of cylindrical functions} \label{sec-reduc2}
Our goal is to prove that for all $B>1$, for ``typical'' \ $\vec{s}$  in the appropriate set, for any  weakly Lipschitz function $f$ on $\Xx_{\bom^+}^{\vec{s}}$, holds
\be \label{Hoeld1}
\sig_f(B(\om,r)) \le C(\bom,\|f\|_L)\cdot r^\gam\ \mbox{for}\ \om\in [B^{-1},B]\ \mbox{and}
 \ 0 < r \le r_0(\bom,\vec{s}, B),
\ee
for some $\gam\in (0,1)$, uniformly in $(x,t)\in \Xx_{\bom^+}^{\vec{s}}$. We will specify $\gam$ at the end of the proof, see (\ref{def-gamma}) and (\ref{defK}).
 The dependence on $\vec{s}$ in the estimate is ``hidden'' in $\sig_f$, the spectral measure of the suspension flow corresponding to the roof function given by $\vec{s}$.
In view of Lemma~\ref{lem-varr}, the estimate (\ref{Hoeld1}) will follow, once we show
\be \label{wts1}
|S^{(x,t)}_R(f,\om)| \le \wtil{C}(\bom,\|f\|_L)\cdot R^{1-\gam/2},\ \mbox{for}\ \om\in [B^{-1},B]\ \mbox{and}\ R\ge R_0(\bom,\vec s, B).
\ee

\begin{lemma} \label{lem-reduc2}
Fix $B>1$. Let $\bom$ be Oseledets regular for the renormalization cocycle $\AA(n)$, vector $\vec s\in \Delta_\bom^{m-1}$, and suppose that for all $\ell\ge \ell_0(\bom,\vec s,B)$ we have
for any bounded cylindrical function $f^{(\ell)}$ of level $\ell$:
\be \label{eq-wts111}
|S^{(x,t)}_R(f^{(\ell)},\om)| \le O_{\bom,\|f^{(\ell)}\|_\infty}(1)\cdot R^{1-\gam/2},\ \mbox{for}\ \om\in [B^{-1},B]\ \mbox{and}\ R\ge e^{\gam^{-1}\theta_1 \ell}.
\ee
Then (\ref{wts1}) holds for any weakly Lipschitz function $f$ on $\Xx_{\bom^+}^{\vec{s}}$.
\end{lemma}

\begin{proof}
Let $f$ be a  weakly Lipschitz function $f$ on $\Xx_{\bom^+}^{\vec{s}}$.
By Lemma~\ref{lem-approx}, we have
$$
{\|f - f^{(\ell)}\|}_\infty \le  {\|f\|}_L\cdot \max_{a\in \Ak} \mu_{\bom^+}(\zeta^{[\ell]}[a]),
$$
for some cylindrical function of level $\ell$, with $\|f^{(\ell)}\|_\infty \le \|f\|_\infty$. By (\ref{eq-measure2}) and the Oseledets Theorem,
since $\bom$ is Oseledets regular, we have
$$
\lim _{n\to \infty} \frac{\log \mu_{\bom^+}(\zeta^{[n]}[a])}{n} = -\theta_1,\ \ \mbox{for all}\  a\in \Ak,
$$
 hence for $\ell\ge \ell_1(\bom)$,
\be \label{approx}
\|f-f^{(\ell)}\|_\infty \le \|f\|_L\cdot e^{-\half\theta_1\ell}.
\ee
Recall that $S_R^{(x,t)}(f,\om) = \int_0^R e^{-2\pi i \om \tau} f\circ h_\tau(x,t)\,d\tau$. 
Let $\ell_2=\ell_2(\bom,\vec s,B):= \max\{\ell_1(\bom), \ell_0(\bom,\vec s,B)\}$ and
$$
R_0=R_0(\bom,\vec s,B) := \exp\bigl[\gam^{-1}\theta_1 \ell_2\bigr].
$$
For $R\ge R_0$ let
\be\label{def-ell}
\ell := \left\lfloor \frac{\gam \log R}{\theta_1}\right\rfloor,
\ee
so that $\ell\ge \ell_2$ and both (\ref{eq-wts111}) and (\ref{approx}) hold. We obtain
$$
|S_R^{(x,t)}(f,\om) - S_R^{(x,t)}(f^{(\ell)},\om)| \le R\cdot \|f\|_L \cdot e^{-\half\theta_1 \ell} \le e^{\theta_1/2}\cdot \|f\|_L\cdot R^{1-\gam/2},
$$
which together with  (\ref{eq-wts111}) imply  (\ref{wts1}).
\end{proof}


\section{Quantitative Veech criterion and the exceptional set}


 By the definition of tile length (\ref{tilength}) and population vector, we have
$$
\|\om|\zeta^{[n]}(v)|_{\vec{s}}\|_{\R/\Z} = \|\langle \vec\ell(v), \om (\Sf^{[n]})^t \vec{s} \rangle \|_{\R/\Z} = \|\langle \vec\ell(v), \AA(n,\bom) (\om\vec{s} )\rangle \|_{\R/\Z}.
$$
For $\vec{x} = (x_1,\ldots,x_m)$ denote by
$$
\|\vec{x}\|_{\R^m/\Z^m} = \max_j \|x_j\|_{\R/\Z},
$$
the distance from $\vec x$ to the nearest integer lattice point in the $\ell^\infty$ metric.

\begin{lemma} \label{lem-lattice}
Let $\{v_j\}_{j=1}^k$ be the good return words for the substitution $\zeta$, such that $\{\vec\ell(v_j)\}_{j=1}^k$ generate $\Z^m$ as a free Abelian group. Then there exists a constant
$C_\zeta>1$ such that
\be \label{ineq-lattice}
C_\zeta^{-1} \|\vec x\|_{\R^m/\Z^m}
\le \max_{j\le k} \|\langle \vec\ell(v_j), \vec{x}\rangle \|_{\R/\Z} \le C_\zeta \|\vec x\|_{\R^m/\Z^m}.
\ee
\end{lemma}

\begin{proof}
Let $\vec{x} = \vec{n} + \vec{\epsilon}$, where $\vec n\in \Z^m$ is the nearest point to $\vec x$. Then $\|\vec x\|_{\R^m/\Z^m}= \|\vec{\epsilon}\|_\infty$, and
$$
\|\langle \vec\ell(v_j), \vec{x}\rangle \|_{\R/\Z} = \|\langle \vec\ell(v_j), \vec{\epsilon}\rangle \|_{\R/\Z} \le \|\vec\ell(v_j)\|_1\cdot \|\vec{\epsilon}\|_\infty,
$$
proving the right inequality in (\ref{ineq-lattice}). On the other hand, the assumption that $\{\vec\ell(v_j)\}_{j=1}^k$ generate $\Z^m$ as a free Abelian group means that for each $j\le m$ there exist $a_{i,j}\in \Z$ such that $\sum_{j=1}^k a_{i,j} \vec\ell(v_j) = \eb_i$, the $i$-th unit vector. Then
$$
\|\vec{x}\|_{\R^m/\Z^m} = \max_i \|x_i\|_{\R/\Z} = \max_i {\left\|\Bigl\langle\sum_{j=1}^k a_{i,j} \vec\ell(v_j), \vec x \Bigr\rangle \right\|}_{\R/\Z} \!\!\!\!\le\ \  \max_{i\le m} \sum_{j=1}^k |a_{i,j}| \cdot
\max_{j\le k} \|\langle \vec\ell(v_j), \vec{x}\rangle \|_{\R/\Z},
$$
finishing the proof.
\end{proof}

In view of (\ref{ineq-lattice}), the product in (\ref{eq-Dioph3}) can be estimated as follows:
\be \label{lattice}
\prod_{\ell+1\le n < \frac{\log R}{4\theta_1}} 
\Bigl( 1 - c_1\cdot \!\!\!\!   \max_{v\in GR(\zeta)}{\bigl\| \om|\zeta^{[n]}(v)|_{\vec{s}}\bigr\|}^2_{\R/\Z}\Bigr)\le \prod_{\ell+1 \le n < \frac{\log R}{4\theta_1}} \left( 1 - \wt c_1\cdot \bigl\|\AA(n,\bom) (\om\vec{s})\bigr\|^2_{\R^m/\Z^m} \right),
\ee
where $\wt c_1\in (0,1)$ is a constant depending only on $\zeta$.

\begin{prop}[Quantitative Veech criterion] \label{prop-quant}
Let $\bom$ be Oseledets regular, $\vec s\in \Delta_{\bom}^{m-1}$, $B>1$, and suppose that there exists $\varrho\in (0,\half)$ such that the set
$
\bigl\{n\in \N:\ \big\|\AA(n,\bom) (\om\vec{s})\bigr\|_{\R^m/\Z^m} \ge \varrho\bigr\}
$
has lower density greater than $\delta>0$ uniformly in $\om\in [B^{-1},B]$, that is,
\be \label{eq-densi}
\card \left\{n\in \N:\ \big\|\AA(n,\bom) (\om\vec{s})\bigr\|_{\R^m/\Z^m} \ge \varrho\right\} \ge \delta N\ \ \mbox{for all}\ \om\in [B^{-1},B]\ \mbox{and}\ N \ge N_0(\bom,\vec s, B).
\ee
Then then H\"older property (\ref{Hoeld1}) holds with 
\be \label{def-gamma}
\gam= \min\left\{\frac{\delta}{16},\frac{-\delta\log(1-\wt c_1\varrho^2)}{8\theta_1}\right\}.
\ee
\end{prop}

\begin{proof}
By Lemma~\ref{lem-reduc2}, it is enough to verify (\ref{eq-wts111}) for a bounded cylindrical function $f^{(\ell)}$, with $\ell \ge \ell_0=\ell_0(\bom,\vec s, B)$.
We use  (\ref{eq-Dioph3}) and (\ref{lattice}), with $\eta = \gam/2$, to obtain:
\be \label{kuku1}
|S^{(x,t)}_R(f^{(\ell)}),\om)| \le  O_{\bom,Q}(1)\cdot \|f^{(\ell)}\|_{_\infty} \left(R^{1/2}+ R^{1+\gam/2}\!\!\!\!\!\!\prod_{\ell+1\le n < \frac{\log R}{4\theta_1}} 
\left( 1 - \wt c_1\cdot \bigl\|\AA(n,\bom) (\om\vec{s})\bigr\|^2_{\R^m/\Z^m} \right)
\right),
\ee
for $\ell\ge \ell_{\gam/2}$ and
 $R\ge e^{8\theta_1\ell}$. 
 Since our goal is (\ref{eq-wts111}), 
we can discard the $R^{1/2}$ term immediately. 

Let $N_0 = N_0(\bom,\vec s, B)$ and 
$$
\ell_0 =\max\left\{ \ell_{\gam/2}, \lceil 4\gam (N_0+1)\rceil\right\}.
$$
For $\ell\ge \ell_0$ take $R\ge e^{\gam^{-1}\theta_1 \ell}$, as required by Lemma~\ref{lem-reduc2}. Then $R\ge e^{8\theta_1\ell}$, since $\gam\le 1/16$, so (\ref{kuku1}) applies. 
Let
$$
N = \left\lfloor\frac{\log R}{4\theta_1}\right\rfloor,
$$
then the choice of $R$ and $\ell_0$ implies that $N\ge N_0$. Thus we have by (\ref{eq-densi}):
\begin{eqnarray}
\prod_{\ell+1\le n < \frac{\log R}{4\theta_1}} 
\left( 1 - \wt c_1\cdot \bigl\|\AA(n,\bom) (\om\vec{s})\bigr\|^2_{\R^m/\Z^m} \right)
& \le & (1-\wt c_1 \varrho^2)^{\delta N-\ell-2} \nonumber \\[-3.3ex]
&  \le & 3 (1-\wt c_1 \varrho^2)^{(\delta \log R)/(8\theta_1)} \nonumber \\
& \le & 3 R^{-\gam}, \label{kuku2}
\end{eqnarray}
where in the second line we used
$$
\delta N - \ell - 2 \le \delta \Bigl[\frac{\log R}{4\theta_1} - 1\Bigr] -\ell-2 \le  \frac{\delta\log R}{4\theta_1} - \ell- 3,
$$
$\ell \le \gam\log R/\theta_1 \le \delta \log R/(16\theta_1)$, and a trivial estimate $(1-\wt c_1 \varrho^2)^{-3} \le 3$ for $\wt c_1\in (0,1),\ \varrho\in (0,\half)$. In the last line we used (\ref{def-gamma}).
Combining (\ref{kuku1}) with (\ref{kuku2}) yields the desired (\ref{eq-wts111}).
\end{proof}

For $\om >0$ and $\vec s\in \Delta_\bom^{m-1}$ let $\vec K_n(\om\vec s)\in \Z^m$ be the nearest integer lattice point to $\AA(n,\bom) (\om\vec{s})$, that is,
\be \label{def-eps}
\AA(n,\bom) (\om\vec{s}) = \vec{K}_n(\om\vec s) + \vec{\eps}_n(\om\vec s),\ \ \|\vec{\eps}_n(\om\vec s)\|_\infty = \bigl\|\AA(n,\bom) (\om\vec{s})\bigr\|_{\R^m/\Z^m}\,.
\ee

\begin{defi}[Definition of the exceptional set]
Given  $\varrho,\delta>0$ and $B>1$, define
$$
E_N(\varrho,\delta,B)  :=\left\{\om\vec{s}\in \R^m_+:\ \vec s\in \Delta_\bom^{m-1},\ \om\in [B^{-1},B],\ 
\card\{n\le N: \|\vec\eps_{n}(\om\vec s)\|_\infty\ge \varrho\} < \delta N\right\},
$$
$$
\Ek_N(\varrho,\delta,B):= \left\{\vec s\in \Delta_\bom^{m-1}:\ \exists\,\om\in [B^{-1},B],\ \om\vec s\in E_N(\rho,\delta,B)\right\},
$$
and 
\be \label{def-excep}
\Fre=\Fre(\varrho,\delta,B):= \bigcap_{N_0=1}^\infty  \bigcup_{N=N_0}^\infty \Ek_N(\varrho,\delta,B).
\ee
\end{defi}

The definition of the exceptional set is related to that in \cite[Section 9]{BuSo18a}; however, here we added an extra step --- the  set $E_N$ of exceptional vectors $\om\vec s$ at scale $N$. The reason is that dimension estimates will focus on the sets $P_\bom^u(E_N)$. On the other hand, it is crucial that the ``final'' exceptional set $\Fre$ be in terms of $\vec s$ in order to obtain uniform H\"older estimates for all $\om\in [B^{-1},B]$, for $\vec s\not\in \Fre$.

\begin{prop} \label{prop-EK} { For any $\epsilon_1>0$ there exist $\delta_0>0$ and $\varrho>0$ such that 
 for $\P$-a.e.\ $\bom\in \Om_\bq$  and any $\delta\in (0,\delta_0)$}, for all $B>1$, and every
Oseledets subspace $H_\bom^J$ corresponding to $\bom$, containing the unstable subspace $E^u_\bom$,
\be \label{eq-dimba}
\dim_H(\Fre(\varrho,\delta,B)\cap H_\bom^J) \le \dim(H_\bom^J) -\kappa + \epsilon_1,
\ee
where $\kappa=\dim(E^u_\bom)$.
\end{prop}

Now we derive Theorem~\ref{th-main11} from Proposition~\ref{prop-EK}, and in the next section we  prove the proposition.

\begin{proof}[Proof of Theorem~\ref{th-main11} assuming Proposition~\ref{prop-EK}] Fix $\epsilon_1>0$ and choose $\rho>0$ and $\delta>0$ such that (\ref{eq-dimba}) holds.
It is enough to verify (\ref{Hoeld1}) and (\ref{wts1})  for all $\vec s\in  \Delta_\bom^{m-1}\setminus \Fre(\rho,\delta,B)$. By definition, $\vec{s} \in   \Delta_{\bom}^{m-1}\setminus \Fre(\rho,\delta,B)$ means that there exists $N_0=N_0(\bom,\vec s, B)\in \N$ such that $\vec s\not\in 
\Ek_N(\varrho,\delta,B)$ for all $N\ge N_0$.
This, in turn, means that for all $\om\in[B^{-1},B]$, we have $\om\vec s\not \in E_N(\varrho,\delta,B)$. However, this is exactly the quantitative Veech criterion (\ref{eq-densi}), and the proof is finished
by an application of Proposition~\ref{prop-quant}. { It is important to note that the H\"older exponent $\gam$ obtained from \eqref{def-gamma} depends only on $\rho$ and $\delta$ and
so does not depend on the (typical) $\bom$.}
\end{proof}


\section{The Erd\H{o}s-Kahane method: proof of Proposition~\ref{prop-EK} }

We now present the  Erd\H{o}s-Kahane argument in {\it vector form}. The argument was introduced by  Erd\H{o}s \cite{Erd} , Kahane \cite{Kahane} for proving Fourier decay for Bernoulli convolutions, see \cite{sixty} for a historical review. Scalar versions of the argument were used in \cite{BuSol2,BuSo18a} to prove H\"older regularity of spectral measures in genus $2$.

In this section we fix a $\P$-generic 2-sided sequence $\bom\in \Om_\bq$.
Under the assumptions of Theorem~\ref{th-main11}, for $\P$-a.e.\ $\bom$, the sequence of substitutions $\zeta(\bom_n)$, 
$n\in \Z$, satisfies several conditions.
To begin with, we  assume that the point $\bom$ is generic for the Oseledets Theorem; that is, assertions (i)-(iii) from Section 3 hold.  We further assume validity of the conclusions
of Proposition~\ref{prop-proba}, and when necessary, we can impose on it additional conditions which hold $\P$-almost surely. 
{\ All implied constants and parameters below (e.g., when writing ``for $n$ sufficiently large'') may depend on $\bom$.}
{\ Recall that $E^u_\bom$ is the unstable Oseledets subspace corresponding to $\bom$, and denote by $E^{cs}_\bom$ the complementary central stable subspace. Let $P^u_\bom$ be the
projection to $E^u_\bom$ along $E^{cs}_\bom$, and similarly, $P^{cs}_\bom = I - P^u_\bom$ the projection to $E^{cs}_\bom$ along $E^u_\bom$.}
 
We defined
$$
\AA(n,\bom) (\om\vec s) = \vec{K}_n(\om\vec s) + \vec{\eps}_n(\om\vec s),
$$
in (\ref{def-eps}), where $\vec K_n(\om\vec s)\in \Z^m$ is the nearest integer lattice point to $\AA(n,\bom) (\om\vec s)$. Below we write
$$
\vec K_n = \vec{K}_n(\om\vec s),\ \ \vec\eps_n = \vec{\eps}_n(\om\vec s),
$$
and $\|\eps_n\| = \|\eps_n\|_\infty$,
to simplify notation.
The idea is that the knowledge of $\vec K_n$ for large $n$ provides a good estimate for the projection of $\om\vec s$ onto the unstable subspace. Indeed, we have
$$
\AA(n,\bom) P_\bom^u(\om\vec s)=P_{\sig^n\bom}^u \AA(n,\bom) (\om\vec s)= P_{\sig^n\bom}^u \vec{K}_n + P_{\sig^n\bom}^u \vec{\eps}_n,
$$
hence
$$
P^u_\bom(\om\vec s) = \AA(n,\bom)^{-1} P_{\sig^n\bom}^u \vec{K}_n + \AA(n,\bom)^{-1} P_{\sig^n\bom}^u \vec{\eps}_n.
$$
By the Oseledets Theorem, we have for $\P$-a.e.\ $\bom$, for any $\epsilon>0$, for all $n$ sufficiently large,
$$
\|\AA(n,\bom)^{-1}P^u_{\sig^n\bom}\|\le e^{-(\theta_\kappa-\epsilon)n},
$$
where $\theta_\kappa>0$ is the smallest positive Lyapunov exponent of the cocycle $\AA$.
{We used that $\|P^u_{\sig^n\bom}\|\le e^{\epsilon n/2}$ for large $n$, since the angle $\angle(E^u_{\sig^n\bom}, E^{cs}_{\sig^n\bom})$ may tend to zero only sub-exponentially.}
{ Fix any $\epsilon\in (0,\theta_\kappa)$ for the rest of the proof.}
By definition
\be \label{up1}
\|\vec{\eps}_n\|\le 1/2<1,\ \ n\ge 0,
\ee
whence
\be \label{good-est}
\|P^u_\bom(\om\vec s) -\AA(n,\bom)^{-1} P_{\sig^n\bom}^u \vec{K}_n\| < e^{-(\theta_\kappa-\epsilon)n},
\ee
for $n$ sufficiently large.

Recall  that we defined $W_n = \log\|A(\bom_n)\|$ in (\ref{def-W}).
Let
\be \label{def-Mn}
M_n:= \bigl(2+\exp(W_{n+1})\bigr)^m\ \ \mbox{and}\ \ \rho_n:= \frac{1/2}{1+\exp(W_{n+1})}\,.
\ee

\begin{lemma} \label{lem-vspom2}
For all $n\ge 0$, we have the following, independent of $\om\vec s\in \R_+^m$:
 
{\bf (i)} Given $\vec{K}_{n}\in\Z^m$, there are at most $M_n$ possibilities for  $\vec K_{n+1}\in\Z^m$;

{\bf (ii)} if $\max\{\|\vec\eps_{n}\|,\|\vec\eps_{{n+1}}\| \}< \rho_n$, then $\vec K_{{n+1}}$ is uniquely determined by $\vec K_{n}$.
\end{lemma}

\begin{proof}
We have by (\ref{def-eps}),
$$
\AA(n,\bom) (\om\vec s) = \vec{K}_n + \vec{\eps}_n,\ \ A(\bom_{n+1}) \AA(n,\bom) (\om\vec s) = \vec{K}_{n+1} + \vec{\eps}_{n+1},
$$
hence
$$
\vec{K}_{n+1} - A(\bom_{n+1})\vec{K}_n =  -\vec{\eps}_{n+1} + A(\bom_{n+1})\vec{\eps}_n.
$$
It follows that
$$
\left\|\vec{K}_{n+1} - A(\bom_{n+1})\vec{K}_n\right\| \le (1 + \exp (W_{n+1}))\max\{\|\vec\eps_{n}\|,\|\vec\eps_{{n+1}}\|\}.
$$
Now both parts of the lemma follow easily.

(i) We have by (\ref{up1}),
$$
\left\|\vec{K}_{n+1} - A(\bom_{n+1})\vec{K}_n\right\| \le (1 + \exp (W_{n+1}))/2=:\Upsilon,
$$
and it remains to note that the $\ell^\infty$ ball of radius $R$ centered at $A(\bom_{n+1})\vec{K}_n$ contains at most $(2\Upsilon+1)^m$ points of the lattice $\Z^m$.

(ii) If $\max\{\|\vec\eps_{n}\|,\|\vec\eps_{{n+1}}\|\}< \rho_n$, then the radius of the ball is less than $\half$, and it contains at most one  point of $\Z^m$, thus $\vec{K}_{n+1} = A(\bom_{n+1})\vec{K}_n$. We are using here that $\Z^m$ is invariant under $A(\bom_{n+1})$ since it is an integer matrix.
\end{proof}

\begin{proof}[Proof of Proposition~\ref{prop-EK}]
Let $\wt E_N(\delta,B)$ be defined by 
\begin{eqnarray*}
\wt E_N(\delta,B)  & :=  & \bigl\{\om\vec s\in \R^m_+:\ \vec s\in \Delta_\bom^{m-1},\ \om\in [B^{-1},B],\\[1.2ex]
& &  \card\{n\le N: \max\{\|\vec\eps_{n}\|,\|\vec\eps_{{n+1}}\|\}\ge \rho_n\} < 
\delta N\bigr\}.
\end{eqnarray*}
We recall  that $\vec\eps_n =\vec\eps_n(\om\vec s)$, but use the shortened notation for simplicity.
First we claim that $\P$-almost surely,
\be \label{claima}
\wt E_N(\delta,B) \supset E_N(\varrho, \delta/4,  B)
\ee
for $N\ge N_0(\bom)$, where
\be \label{defK}
\varrho = \frac{1/2}{1+e^K}, \ \ \mbox{with}\ \ K =2 L_1 \log(1/\delta).
\ee
Here  $L_1$ is from Proposition~\ref{prop-proba}. { Note that $\rho$ depends only on $\delta$, but not on $\bom$; we assume throughout that $\bom$ is such that \eqref{w-cond2} holds.}

Suppose $\om\vec{s} \not\in \wt E_N(\delta,B)$. Then
there exists a subset $\Gam_N \subset \{1,\ldots,N\}$ of cardinality $\ge \delta N$ such that
$$
 \max\{\|\vec\eps_{n}\|,\|\vec\eps_{{n+1}}\|\}\ge \rho_n\ \ \ \mbox{for all}\ n \in \Gam_N.
$$
Note that $\rho_n < \varrho$ is equivalent to $W_{n+1} >K$ by (\ref{defK}) and (\ref{def-Mn}).
Observe that 
 there are fewer than $\delta N/2$ integers $n\le N$ for which $W_{n+1}>K$, for $N\ge N_0(\bom)$ by Corollary~\ref{cor-immed}. Thus
$$
\card\bigl\{n\le N:\  \max\{\|\vec\eps_{n}\|,\|\vec\eps_{{n+1}}\|\}\ge \varrho \bigr\}\ge \delta N/2,
$$
hence there are at least $\delta N/4$ integers $n\le N$ with $\|\vec\eps_n\|\ge \varrho$, and so
$\om\vec{s}\not \in E_N(\varrho, \delta/4, B)$ which confirms (\ref{claima}).

It follows that it is enough to show that if $\delta>0$ is sufficiently small,  then for every Oseledets subspace $H_\bom^J\supset E^u_\bom$,
$$
\dim_H(\wtil\Fre(\delta,B)\cap H_\bom^J) \le \dim(H_\bom^J) - \kappa + \epsilon_1=:\beta.
$$
 where
$$
\wtil\Fre(\delta,B):=\bigcap_{N_0=1}^\infty \bigcup_{N=N_0}^\infty \wtil \Ek_N(\delta,B),
$$
$$
\wtil \Ek_N(\delta,B) := \left\{\vec s\in \Delta_\bom^{m-1}:\ \exists\,\om\in[B^{-1},B],\ \om\vec s\in \wtil E_N(\delta,B)\right\}.
$$
Let $\Hk^\beta$ denote the $\beta$-dimensional Hausdorff measure. A standard method to prove 
\be \label{dimen}
\Hk^\beta(\wtil\Fre(\delta,B)\cap H_\bom^J) <\infty,
\ee
 whence $\dim_H(\wtil\Fre(\delta,B)\cap H_\bom^J)\le \beta$, 
is to show that $\wtil \Ek_N(\delta,B)\cap H_\bom^J$ may be covered by $\approx e^{\beta\cdot N\alpha }$ balls of radius $\approx e^{-N\alpha}$ for $N$ sufficiently large, for some $\alpha>0$.
Observe that the map from $\wtil \Ek_N(\delta,B)$ to $\wtil E_N(\delta,B)$ is simply $\om\vec s\mapsto \vec s$ for $\om\in [B^{-1},B]$ and $\vec s\in \Delta_\bom^{m-1}$, which is  Lipschitz, with a Lipschitz constant $O_{\bom,B}(1)$; thus it is enough to produce a covering of $\wtil E_N(\delta,B)\cap H_\bom^J$.

{\ Denote $H_\bom^J \ominus E_\bom^u = H^J_{\bom}\cap E^{sc}_\bom$.}
By assumption, $H_\bom^J = E^u_\bom \oplus (H_\bom^J \ominus E_\bom^u)$, hence
\be \label{eq-prod}
\wtil E_N(\delta,B)\cap H_\bom^J \subset P_\bom^u\left(\wtil E_N(\delta,B)\right) \oplus \Bigl\{\vec y\in H_\bom^J \ominus E_\bom^u:\ \|\vec y\|_1 \le \frac{B\|P_\bom^{cs}\|}{\min_a\mu_{\bom^+}([a])}\Bigr\}=: F_1\oplus F_2.
\ee
{\ Here $\oplus$ stands for the direct sum decomposition corresponding to $\bom$. Denote by $\Ns(F,r)$ the minimal number of balls of radius $r$ needed to cover a set $F$. We have
\be \label{eq-produ2}
\Ns(F_1\oplus F_2,r) \le O_\bom(1)\cdot \Ns(F_1,r)\cdot \Ns(F_2,r),\ \ \mbox{for any}\ r>0,
\ee
since $F_1\oplus F_2$ is bi-Lipschitz equivalent to $F_1\times F_2$. The Lipschitz constant depends on the angle between $F_1$ and $F_2$ and thus only depends on $\ba$.
}
Observe that $\dim(H_\bom^J \ominus E_\bom^u)=\dim(H_\bom^J)-\kappa$, so 
{\
\be \label{cover1}
\Ns(F_2,e^{-\alpha N}) \le O_{\bom,B}(1)\cdot \exp\left[\alpha N (\dim(H_\bom^J)-\kappa)\right]\ \  \mbox{for any}\ \alpha>0.
\ee
}
Thus it remains to produce a covering of $F_1=P_\bom^u\left(\wtil E_N(\delta,B)\right)$.
 Suppose $\om\vec{s} \in \wt E_N(\delta,B)$ and find the 
corresponding sequence $\vec K_{n}, \vec\eps_{n}$ from (\ref{def-eps}). We have from (\ref{good-est}) that for $N$ sufficiently large,
\be \label{eqrad}
P_\bom^u(\om\vec s)\ \ \mbox{is in the ball centered at}\ \  \AA(N,\bom)^{-1} P_{\sig^N\bom}^u \vec{K}_N\ \ \mbox{of radius}\ \ e^{-(\theta_\kappa-\epsilon)N}.
\ee
 Since $\bom$ is fixed, it is enough to estimate 
 the number of sequences $\vec K_{n}$, $n\le N$, which may arise.

 Let $\Psi_N$ be the set of $n\in \{1,\ldots,N\}$ for which
we have $\max\{\|\vec\eps_{n}\|,\|\vec\eps_{{n+1}}\|\}\ge \rho_n$.
By the definition of $\wt E_N(\delta,B)$ we have $|\Psi_N| <\delta N$. 
There are $\sum_{i< \delta N} {N\choose i}$ such sets. For a fixed $\Psi_N$ the number of possible sequences $\{\vec K_{n}\}$ is at most
$$
\Bk_N:= \prod_{n\in \Psi_N} M_n,
$$
times the number of ``beginnings'' $\vec K_{0}$, by Lemma~\ref{lem-vspom2}.
The number of possible $\vec K_{0}$ is  bounded, independent of $N$, by a constant depending on $B$, since $\om\in [B^{-1},B]$ and $\vec s\in \Delta_\bom^{m-1}$.
In view of  $M_n \le 3^m e^{mW_{n+1}}$, see (\ref{def-Mn}), we have by (\ref{w-cond2}),  for $N$ sufficiently large, 
$$
\Bk_N \le O_{\bom,B}(1) \cdot 3^{m\delta N}  \exp\left(m \sum_{n\in \Psi_N}W_{n+1}\right)\le O_{\bom,B}(1) \cdot 3^{m\delta N} \exp\left[m\cdot L_1\log(1/\delta)(\delta N)\right].
$$
Thus, by (\ref{eqrad}), 
{\ $\Ns(F_1, e^{-(\theta_\kappa-\epsilon)N}) = \Ns\bigl(P_\bom^u(\wtil E_N(\delta,B)),e^{-(\theta_\kappa-\epsilon)N}\bigr)$ is not greater than}
\be \label{lasta1}
O_{\bom,B}(1)\cdot  \sum_{i<\delta N} {N\choose i} \cdot 3^{m\delta N} \exp\left[m\cdot L_1\log(1/\delta)(\delta N)\right]  \le O_{\bom,B}(1)\cdot \exp\left[L_2 \log(1/\delta)(\delta N)\right],
\ee
for some $L_2>0$, using the standard entropy estimate (or Stirling formula) for the binomial coefficients. Since $\delta\log(1/\delta)\to 0$ as $\delta\to 0$, we can choose $\delta_0>0$ so small 
that $\delta< \delta_0$ implies
$$
\left[L_2\log(1/\delta)(\delta N)\right] < \epsilon_1 (\theta_\kappa-\epsilon)N.
$$
Combining this with (\ref{eq-prod}), \eqref{eq-produ2}, and (\ref{cover1}), we obtain that $\wt E_N(\delta,B)\cap H_\bom^J$ may be covered by
$$
O_{\bom,B}(1)\cdot\exp\left[(\theta_\kappa-\eps)N\beta\right]\ \ \mbox{balls of radius}\ \ e^{-(\theta_\kappa-\epsilon)N}
$$
for $N$ sufficiently large, where $\beta = \dim(H_\bom^J) - \kappa + \epsilon_1$. This confirms (\ref{dimen}).
The proof of Proposition~\ref{prop-EK}, and hence of
Theorem~\ref{th-main11}, is now complete.
\end{proof}


\section{Derivation of  Theorem~\ref{main-moduli}  and Theorem~\ref{th-twisted} from Theorem~\ref{th-main1} } \label{sec-deriv}

This section is parallel to \cite[Section 11]{BuSo18a}; however, we need to make a number of changes, in view of the requirements on the word $\bq$.

Recall the discussion in Section~\ref{subsec-symbol} and Remark~\ref{remark-main}.
Consider our surface $M$ of genus $g\ge 2$. By the results of \cite[Section 4]{Buf-umn} there is a correspondence between almost every translation flow with respect to the Masur-Veech measure and a natural flow on a ``random'' 2-sided Markov compactum, which is, in turn, measurably isomorphic to the suspension flow over a one-sided Markov compactum, or, equivalently, an $S$-adic system $X_{\bom^+}$ for $\bom^+\in \FrA$. The roof function of the suspension flow is piecewise constant, depending only on the first symbol, and we can express it as a vector of ``heights''. This symbolic realization uses Veech's construction \cite{veech} of the space of zippered rectangles which corresponds to a connected component of a stratum $\Hk$. Given a Rauzy class $\Rk$, we get a space of $S$-adic systems on $m\ge 2g$ symbols, which provide a symbolic realization of the interval exchange transformations (IET's) from $\Rk$. As shown by Veech \cite{veechamj}, the ``vector of heights'' obtained in this construction necessarily belongs to a subspace $H(\pi)$, which is invariant under the Rauzy-Veech cocycle and depends only on the permutation of the IET. In fact, the subspace $H(\pi)$ has dimension $2g$ and is the sum of the stable and unstable subspaces for the Rauzy-Veech cocycle. By the  result of Forni \cite{Forni}, there are $g$ positive Lyapunov exponents, thus $\dim(E^u_\bom)=g\ge 2$. In the setting of Theorem~\ref{th-main1} we will take $H(\pi) = H_\bom(\pi)$ to be our Oseledets subspace $H^J_\bom$, containing the strong unstable subspace $E^u_\bom$.

Theorem~\ref{main-moduli} (in the expanded form) will follow from Theorem~\ref{th-main1} by the argument given in Remark~\ref{remark-main}. Indeed, by the assumption, for $\mu$-a.e.\ Abelian differential, the induced conditional measure on a.e. fibre has Hausdorff dimension $\ge 2g-\kappa + \delta$, hence taking $\epsilon_1=\delta$ and using the estimate (\ref{dimesta}) for the exceptional set, we see that exceptional set has zero $\mu$-measure.


Thus it remains to check that the conditions of Theorem~\ref{th-main1} are satisfied. The property (C1) holds because the renormalization matrices in the Rauzy-Veech induction all have determinant $\pm 1$. Condition (C2) holds by a theorem of Zorich \cite{Zorich}. As already mentioned, property (a) (on Lyapunov exponents) in Theorem~\ref{th-main1} holds by a theorem of Forni \cite{Forni}. 

Next we explain how to achieve the combinatorial properties (b) and (c) of a word $\bq\in \FrA^k$. Recall the discussion of Rauzy induction in Section~\ref{subsec-Rauzy}, which we repeat in part for convenience.
As is well-known, for almost every IET, there is a corresponding infinite path in the Rauzy graph, and the length of the interval on which we induce tends to zero. For any finite ``block'' of this path, we
have a pair of intervals $J\subset I$ and IET's on them, denoted $T_I$ and $T_J$, such that both are
exchanges of $m$ intervals and $T_J$ is the first return map of $T_I$ to $J$. Let $I_1, \dots, I_m$ be the subintervals of the exchange $T_I$ and $J_1, \dots, J_m$ the subintervals of the exchange $T_J$. 
Let $r_i$ be the return time for the interval $J_i$ into $J$ under $T_I$, that is,
$
r_i=\min\{ k>0: T_I^kJ_i\subset J\}.
$
Represent $I$ as a Rokhlin tower over the subset $J$ and its induced map $T_J$, and  write 
$$I=\bigsqcup\limits_{i=1,\dots, m, k=0, \dots, r_i-1} T^{k}J_i.
$$
By construction, each of the ``floors'' of our tower, that is, each of the subintervals $T_I^{k}J_i$, is a subset of some, of course, unique,  subinterval  of the initial exchange, and we define an integer $n(i,k)$ by the formula
$$
T_I^{k}J_i\subset I_{n(i,k)}.
$$
To the pair $I,J$ we now assign a substitution $\zeta_{IJ}$ on the alphabet 
$\{1, \dots, m\}$ by the formula 
\begin{equation} \label{zeta-def}
\zeta_{IJ}: i\to n(i,0)n(i,1)\dots n(i, r_i-1).
\end{equation}
Words obtained form finite paths in the Rauzy graph will be called {\em admissible} {\ remove: (this agrees with the notion of ``admissible word'' in Section 3)}. By results of Veech, every admissible word appears infinitely often with positive probability in a typical infinite path. Condition (c) of Theorem~\ref{th-main1} and the simplicity of the admissible word from (b) are verified in the next lemma.
\begin{lemma} \label{lem-combi1}
There exists an admissible word $\bq$, which is {\bf simple}, whose associated matrix $A(\q)$ has strictly positive entries, and the corresponding substitution $\zeta$, with $Q = \Sf_\zeta=A(\q)^t$ having the property that there exist  {\bf good return words}
$u_1,\ldots,u_m\in GR(\zeta)$, such that $\{\vec{\ell}(u_j):\ j\le m\}$ generate the entire 
$\Z^m$ as a free Abelian group.
\end{lemma}
\begin{proof}
Recall that by ``word'' here we mean a sequence of substitutions corresponding to a finite path in the Rauzy graph. Denote by $\zeta_V$ the substitution corresponding to a path $V$.
The
alphabet may be identified with the set of edges. By construction, if we concatenate two paths $V_1 V_2$, we obtain
$
\zeta_{V_1 V_2} = \zeta_{V_2} \zeta_{V_1}.
$
First we claim that there exists a loop $V$ in the Rauzy graph, such that $A(V)$ is strictly positive and $\zeta_V(j)$ starts with the same letter $c=1$ for every $j\le m$.
Indeed, start with an arbitrary loop $V$ in the Rauzy graph such that the corresponding renormalization matrix has all entries positive. Consider the  interval exchange transformation with periodic 
Rauzy-Veech expansion obtained by going along the loop repeatedly (it is known from \cite{veech} that such an IET exists).
As the number of passages through the loop grows, the length of the interval forming phase space of the new interval exchange (the result of the induction process) goes to zero. In particular, after sufficiently many moves, this interval will be completely contained in the first subinterval of the initial interval exchange --- but this  means, in view of (\ref{zeta-def}) that $n(i,0)=1$ for all $i$, and hence the resulting substitution
 $\zeta_{V^n}$
 has the property that 
 $\zeta_{V^n}(j)$
 starts with $c=1$ for all $j$.
 
Next, observe that for any loop $W$ starting at the same vertex, the substitution
$\zeta=
\zeta_{WV^{2n}} = \zeta_{V^n}\zeta_{V^n}\zeta_W$ has the property that every $\zeta_{V^n}(j)$ is a return word for it. Indeed, applying $\zeta$ to any sequence we obtain a concatenation of words of the form $\zeta_{V^n}(j)$ for $j\le m$, in some order, and every one of them starts with $c=1$. Therefore, they are all return words. Moreover, every letter $j$ appears in every
word $\zeta_{V^n}(i)$, since the substitution matrix of $\zeta_V$ is strictly positive. Thus,
every word  $u_j:=
\zeta_{V^n}(j)$ appears in every $\zeta(i)$, which means that all these words are good return words for $\zeta$. The corresponding population vectors 
$\vec\ell(u_j)$ are the columns of the substitution matrix of $\zeta_{V^n}$. As is well-known, the matrices corresponding to Rauzy operations are invertible and unimodular, which means that the columns of $\Sf_{\zeta_{V^n}}$ are linearly independent and generate $\Z^m$ as a free Abelian group.
It remains to choose $W$ to make sure that the word $WV^{2n}$ is simple. It is known that in the Rauzy graph there are $a$- and $b$-cycles starting at every vertex. Assume that the loop $V$ ends with an edge labelled by $a$ (the other cases is treated similarly). Then first fix an $a$-loop $W_2$ starting at the same vertex as $V$, so that $W_2 V^{2n}$ starts and ends with an $a$-edge. Then choose a $b$-loop $W_1$ starting at the same vertex, with the property that
$|W_1|> |W_2 V^{2n}|$. We will then consider the admissible word $W_1 W_2 V^{2n}$, and we claim that it is simple. Indeed, the word in the alphabet $\{a,b\}$ corresponding to it, has the form $b^k a\ldots a$, and it is simple, because its length is less than $2k$. The proof is complete.
\end{proof}
Condition (\ref{eq-star1}), a variant of the exponential estimate for return times of the Teichm\"uller flow into compact sets, is proved for an arbitrary genus $g\ge 2$ in \cite[Prop.\,11.3]{BuSo18a} by modifying an 
argument from \cite{buf-jams}.  Theorem~\ref{main-moduli}  and Theorem~\ref{th-twisted} are proved completely.



\begin{thebibliography}{99}



\bibitem{AF}Avila, Artur; Forni, Giovanni. Weak mixing for interval exchange transformations and translation flows.
{\em Annals of Mathematics} {\bf 165} (2007), 637--664. 




 \bibitem{barpes} Barreira, Luis; Pesin, Yakov. Nonuniform hyperbolicity. Dynamics of systems with nonzero Lyapunov exponents. {\em Encyclopedia of Mathematics and its Applications}, 115. Cambridge University Press  2007. 

\bibitem{BD} Berth\'e, Valerie and Vincent Delecroix,
Beyond substitutive dynamical systems: $S$-adic expansions. Numeration and substitution 2012, 81--123, RIMS K\^{o}ky\^{u}roku Bessatsu, B46, Res. Inst. Math. Sci. (RIMS), Kyoto, 2012.

\bibitem{BST_2019}  Berth\'e, Val\'erie; Steiner, Wolfgang; Thuswaldner, J\"org M.
Geometry, dynamics, and arithmetic of $S$-adic shifts. {\em Ann.\ Inst.\ Fourier (Grenoble)} {\bf 69} (2019), no.\ 3, 1347--1409.

\bibitem{BSTY} Berth\'e, Val\'erie; Steiner, Wolfgang; Thuswaldner, J\"org M.; Yassawi, Reem; Recognizability for sequences of morphisms. {\em Ergodic Theory Dynam. Systems} {\bf  39} (2019), no.\ 11, 2896--2931.

\bibitem{buf-jams} Bufetov, Alexander I.
Decay of correlations for the Rauzy--Veech--Zorich induction map 
on the space of interval exchange transformations and the 
Central Limit Theorem for the Teichm{\"u}ller Flow 
on the moduli space of Abelian differentials. {\em Journal of the American Mathematical Society} {\bf 19} (2006), no.\ 3, 579--623.

\bibitem {Buf-umn} Bufetov,  Alexander I. Limit theorems for special
flows over Vershik's automorphisms.
{\em Russian Mathematical Surveys} {\bf 68} (2013), no.\ 5, 789--860.

\bibitem{Bufetov1} Bufetov, Alexander I. Limit theorems for translation flows. {\em Annals of Mathematics} {\bf 179} (2014), no.\ 2, 431--499.

\bibitem{BufGur} Bufetov,   Alexander I.; B. M. Gurevich, Boris M. Existence and uniqueness of the measure of maximal entropy for the Teichm\"uller flow on the moduli space of Abelian differentials. {\em Sbornik: Mathematics} {\bf 202} (2011), no.\ 7, 935--970.

\bibitem{BuSol--limit} Bufetov, Alexander I.; Solomyak, Boris.
Limit theorems for self-similar tilings,
{\em Comm.\ Math.\ Physics} {\bf 319} (2013), no.\ 3, 761--789.


\bibitem{BuSol2} Bufetov, Alexander I.; Solomyak, Boris.
On the modulus of continuity for spectral measures in substitution dynamics. {\em Advances in Mathematics} {\bf 260} (2014), 84--129.

\bibitem{BuSo18a} Bufetov, Alexander I.; Solomyak, Boris.
The H\"older property for the spectrum of translation flows in genus two, {\em Israel J.\ Math.} {\bf 223} (2018), no.\ 1, 205--259.

\bibitem{BuSo18b} Bufetov, Alexander I.; Solomyak, Boris.
On ergodic averages for parabolic product flows, {\em Bull.\ Soc.\ Math.\ France} {\bf 146} (2018), no.\ 4, 613--628.

\bibitem{BuSo19}  Bufetov, Alexander I.; Solomyak, Boris. A spectral cocycle for substitution systems and translation flows. 
{\em J. d'Analyse Math.}, to appear.
{\em Preprint} arXiv: 1802.04783.


\bibitem{Erd} Erd\H{o}s, Paul.
On the smoothness properties of Bernoulli convolutions.
{\em Amer.\ J.\ Math.} {\bf 62} (1940), 180--186.


\bibitem{Falc-book} Falconer, Kenneth J. {\em Techniques in fractal geometry}, John Wiley \& Sons, 1997.

\bibitem{Fogg} Pytheas N. Fogg,  {\em Substitutions in dynamics, arithmetics and combinatorics}, Edited by V. Berth\'e, S. Ferenczi, C. Mauduit and A. Siegel.
Lecture Notes in Math., 1794, Springer, Berlin, 2002. 

\bibitem {Forni}
Forni, Giovanni. Deviation of ergodic averages for area-preserving
flows on surfaces of higher genus. {\em Annals of Mathematics} (2) {\bf 155} (2002), no.\ 1, 1--103.

\bibitem{Forni2}
Forni, Giovanni. Twisted translation flows and effective weak mixing. {\em Preprint} arXiv:1908.11040. 


\bibitem{furst} Furstenberg, Harry. Stationary processes and prediction theory. {\em Annals of Mathematics Studies}, No. 44. Princeton University Press, Princeton, N.J. 1960.


\bibitem{GJ} Gjerde, Richard; Johansen, {\O}rjan. Bratteli-Vershik models for Cantor minimal systems associated to interval exchange transformations. {\em Math. Scand.} {\bf 90} (2002), no.1, 87--100.

\bibitem{Ito} Ito, Shunji.
A construction of transversal flows for maximal Markov automorphisms.
{\em Tokyo Journal of Mathematics} {\bf 1} (1978), no. 2, 305--324.

\bibitem{Kahane}  Kahane, Jean-Pierre. Sur la distribution de certaines s\'eries al\'eatoires. {\em
Colloque Th.\ Nombres [1969, Bordeaux],  Bull.\ Soc.\ math.\ France}, 
 M\'{e}moire 25 (1971), 119--122.

\bibitem{katok} Katok, Anatole.
Interval exchange transformations and some special flows are not mixing.
{\em Israel Journal of Mathematics} {\bf 35} (1980), no.\ 4, 301--310.

\bibitem{LinTre} Lindsey, Kathryn; Trevi\~no, Rodrigo.
Infinite type flat surface models of ergodic systems. 
{\em Discrete Contin. Dyn. Syst.} {\bf 36} (2016), no.\ 10, 5509--5553.

\bibitem{Liv1} Livshits, Alexander N.
Sufficient conditions for weak mixing of substitutions and of stationary adic transformations. (Russian)
{\em Mat.\ Zametki} {\bf 44} (1988), no. 6, 785--793, 862; translation in
{\em Math.\ Notes} {\bf 44} (1988), no. 5--6, 920--925 (1989)

\bibitem{MMY}
Marmi, Stefano; Moussa, Pierre; Yoccoz, Jean-Christophe. 
The cohomological equation for Roth-type interval exchange maps. {\em
J.\ Amer.\ Math.\ Soc.} {\bf 18} (2005), no.\ 4, 823--872.

\bibitem{masur} Masur, Howard A.
Interval exchange transformations and measured foliations.
{\em Annals of  Mathematics}  {\bf 115}  (1982),  1, 169--200.

\bibitem{Mosse} Moss\'e, Brigitte. Puissances de mots et reconnaissabilit\'e des points fixes d'une substitution, {\em Theoret.\ Comput.\ Sci.} 99 (1992), no.\ 2, 327--334.

\bibitem{oseledets}
 Oseledets, Valery I. A multiplicative ergodic theorem. Characteristic Lyapunov exponents of dynamical systems. {\em  Trudy Mosk. Mat. Obs./Proceedings of the Moscow Mathematical Society} {\bf 19} (1968), 179--210.

\bibitem{sixty}  Peres, Yuval;  Schlag, Wilhelm; Solomyak, Boris.
Sixty years of Bernoulli convolutions,
{\em Fractal Geometry and Stochastics} {\bf II},
C. Bandt, S. Graf, and M.\ Z\"ahle (editors), Progress
in Probability Vol.\ {\bf 46}, 39--65, Birkh\"{a}user, 2000.

\bibitem{Queff} Queffelec, Martine. {\em Substitution Dynamical Systems - Spectral Analysis}.  LNM., vol.\ 1294, Springer,  2010.

\bibitem{Rauzy} 
Rauzy, G\'erard.
\'Echanges d'intervalles et transformations induites. 
{\em Acta Arith.} {\bf 34} (1979), no.\ 4, 315--328.






\bibitem{solomyak}
Solomyak, Boris. On the spectral theory of adic transformations.
Representation theory and dynamical systems,  217--230, {\em Adv. Soviet Math.} {\bf 9},
Amer. Math. Soc., Providence, RI, 1992.

\bibitem{Tre_prepint}
Trevi\~no, Rodrigo. Quantitative weak mixing for random substitution tilings.
{\em Preprint} arXiv:2006.16980.

\bibitem{Veech0}
Veech, William A.
Interval exchange transformations.
{\em J.\ Analyse Math.} {\bf 33} (1978), 222--272.


\bibitem {veech}
Veech, William A.
Gauss measures for transformations on the space of interval exchange maps.
{\em Annals of Mathematics} (2) {\bf 115} (1982), no.\ 1, 201--242.

\bibitem {veechamj} Veech, William A.
The metric theory of interval exchange transformations. I. Generic spectral properties.  {\em Amer. J.  Math.}  {\bf 106}  (1984),  no. 6, 1331--1359.








\bibitem{V3}
Veech, William A. Moduli spaces of quadratic differentials.  {\em J. Analyse Math. } 55  (1990), 117--171.


\bibitem  {Vershik1}
Vershik, Anatoli\u{i}  M.
A theorem on Markov periodic approximation in ergodic theory. (Russian)
Boundary value problems of mathematical physics and related questions in the theory
of functions, 14.
{\em Zap. Nauchn. Sem. St.-Peterb. Otdel. Mat. Inst. Steklov}. (LOMI) 115 (1982), 72--82,
306.

\bibitem {Vershik2}
Vershik, Anatoli\u{i} M.
 The adic realizations of the ergodic actions with the homeomorphisms of
 the Markov compact and the ordered Bratteli diagrams. {\em  Zap. Nauchn. Sem.
 S.-Peterburg. Otdel. Mat. Inst. Steklov. (POMI)} 223 (1995), Teor.
 Predstav. Din. Sistemy, Kombin. i Algoritm. Metody. I, 120--126, 338;
 translation in J. Math. Sci. (New York) 87 (1997), no. 6, 4054--4058.

\bibitem {Vershik-Livshits}
 Vershik, Anatoli\u{i} M.; Livshits, Alexander N.
Adic models of ergodic transformations, spectral theory,
substitutions, and related topics. Representation theory and
dynamical systems, 185--204. {\em Adv. Soviet Math.} {\bf 9},
Amer. Math. Soc., Providence, RI, 1992.

\bibitem{Viana-IMPA} Viana Marcelo.
Lyapunov exponents of Teichm{\"u}ller flows, {\em Preprint IMPA}, 2006.

\bibitem{viana2} Viana Marcelo. Lectures on 
Interval Exchange Transformations and Teichm\"uller Flows, {\em Preprint IMPA}, 2008.

\bibitem{yoccoz}
Yoccoz, Jean-Christophe.
Interval exchange maps and translation surfaces. In:
{\em Homogeneous flows, moduli spaces and arithmetic}, 1--69,
Clay Math. Proc., 10, Amer. Math. Soc., Providence, RI, 2010.


\bibitem{Zorich} Zorich, Anton. 
Finite Gauss measure on the space of interval exchange transformations.
{\em Ann.\ Inst.\ Fourier (Grenoble)} {\bf 46} (1996), 325--370.

\bibitem{Zorich1}
Zorich, Anton.
Flat surfaces. In: {\em Frontiers in number theory, physics, and geometry. I}, 437--583, Springer, Berlin, 2006.






\end{thebibliography}
\end{document}